\documentclass[11pt]{article}

\usepackage{graphicx}
\usepackage{amsmath}
\usepackage{amsthm}
\usepackage{amsfonts}
\usepackage{amssymb, amscd}

\usepackage[all, knot]{xy}

\newtheorem{thm}{Theorem}[section]
\newtheorem{cor}[thm]{Corollary}
\newtheorem{lem}[thm]{Lemma}
\newtheorem{prop}[thm]{Proposition}
\newtheorem{example}[thm]{Example}
\theoremstyle{definition}
\newtheorem{defn}[thm]{Definition}
\theoremstyle{remark}
\newtheorem{rem}[thm]{Remark}
\numberwithin{equation}{section}

\input{epsf.sty}

\begin{document}

\title{Cohomology of  Hom-Lie superalgebras \\ and $q$-deformed Witt superalgebra}

\author{Faouzi AMMAR
\and Abdenacer MAKHLOUF
\and Nejib SAADAOUI }

\maketitle

\begin{abstract}
The purpose of this paper is to define the representation and the cohomology of  Hom-Lie superalgebras. Moreover we study Central extensions and provide as application the computations of the derivations and second cohomology group  of $q$-deformed Witt superalgebra.
\end{abstract}



\section*{Introduction}
Hom-Lie algebras and other Hom-algebras structures have been widely investigated these last years. They were introduced and studied in \cite{HLS,LS1,LS2,LS3,LSgraded} motivated initially by examples of deformed Lie algebras coming from
twisted discretizations of vector fields. The paradigmatic examples are $q$-deformations of Witt and Virasoro algebras based on $\sigma$-derivation \cite{AizawaSaito,Hu,Liu,HLS}.
 Hom-Lie superalgebras were studied in \cite{AmmarMakhloufJA2010}. Cohomology theory of  Hom-Lie algebras was studied in    \cite{AEM,MS3,sheng}.
The purpose of  this paper is to study  representations and  cohomology of Hom-Lie superalgebras. As application,  we provide some calculations for $q$-deformed Witt superalgebra.

The paper is organized as follows. In the first section we give the definitions and some key constructions of Hom-Lie superalgebras. Section $2$ is dedicated to the representation theory 
Hom-Lie superalgebra including adjoint and coadjoint representation. In Section $3$ we construct   family of  cohomologies of Hom-Lie superalgebras. In Section $4$, we discuss extensions of Hom-Lie superalgebras  and their connection to cohomology.  In the last section we compute the derivations  and scalar second  cohomology group of the $q$-deformed Witt superalgebra.

\section{Hom-Lie superalgebras}

In this section,  we review the theory of Hom-Lis superalgebras established in \cite{AmmarMakhloufJA2010} and generalize some results of \cite{Said}.

Let  $\mathcal{G}$ be a linear superspace over a field $\mathbb{K}$ that is a $\mathbb{Z}_{2}$-graded linear space with a direct sum $\mathcal{G}=\mathcal{G}_{0}\oplus \mathcal{G}_{1}.$\\
The elements of $\mathcal{G}_{j}$, $j\in \mathbb{Z}_{2},$ are said to be homogenous  of parity $j.$ The parity of  a homogeneous element $x$ is denoted by $|x|.$\\
The space $End (\mathcal{G})$ is $\mathbb{Z}_{2}$ graded with a direct sum $End (\mathcal{G})=(End (\mathcal{G}))_{0}\oplus(End (\mathcal{G}))_{1}$ where
$(End (\mathcal{G}))_{j}=\{f\in End (\mathcal{G}) /f (\mathcal{G}_{i})\subset \mathcal{G}_{i+j}\}.$
The elements of $(End (\mathcal{G}))_{j}$  are said to be homogenous of parity $j.$

\begin{defn}
A Hom-Lie superalgebra  is a triple $(\mathcal{G},\ [.,.],\ \alpha)$\ consisting of a superspace $\mathcal{G}$, an even bilinear map \ $\ [.,.]:\mathcal{G}\times \mathcal{G}\rightarrow \mathcal{G}$ \ and an even superspace homomorphism \ $ \alpha:\mathcal{G}\rightarrow \mathcal{G} \ $satisfying
\begin{eqnarray}
&&[x,y]=-(-1)^{|x||y|}[y,x],\\
&&(-1)^{|x||z|}[\alpha(x),[y,z]]+(-1)^{|z||y|} [\alpha(z),[x,y]]+(-1)^{|y||x|} [\alpha(y),[z,x]]=0\label{jacobie}
\end{eqnarray}
for all homogeneous element x, y, z in $\mathcal{G}.$\\
Let $\left( \mathcal{G}, [\cdot, \cdot], ,\alpha \right) $ and $\left(
\mathcal{G}^{\prime }, [\cdot, \cdot] ^{\prime },\alpha^{\prime
}\right) $ be two Hom-Lie superalgebras. An even homomorphism\\
 $f\ :\mathcal{G}\rightarrow \mathcal{G}^{\prime }$ is said
to be a
\emph{morphism of  Hom-Lie superalgebras} if
\begin{eqnarray} [f(x), f(y)] ^{\prime }&=&f ([x, y])\quad \forall x,y\in \mathcal{G}
\\ f\circ \alpha&=&\alpha^{\prime
}\circ f.
\end{eqnarray}

\end{defn}
\begin{rem}
We recover the classical Lie superalgebra when $\alpha=id.$\\
The Hom-Lie algebra are obtained when the part of parity one is trivial.
\end{rem}
\begin{example}
Let  $\mathcal{G}=\mathcal{G}_0 \oplus \mathcal{G}_1$ be a $3$-dimensional superspace where $\mathcal{G}_0$ is generated by $e_1 $ and $\mathcal{G}_1$ is generated by $ e_2,\ e_3.$
The triple $(\mathcal{G},[.,.],\alpha)$ is  Hom-Lie superalgebra
 defined by\\ $[e_1 ,e_2]=2e_2 ,[e_1 , e_3]=2e_3$ and  $[e_2,e_3]=e_1$, with  $\alpha(e_1)=e_1,\ \alpha(e_2)=e_3, \ \alpha(e_3)=-e_2. $
\end{example}

\begin{defn}
Let $(\mathcal{G},[.,.],\alpha)$ be a Hom-Lie superalgebra. A Hom-Lie superalgebra is called
\begin{itemize}
\item multiplicative  if $\forall x,y\in \mathcal{G}$ we have $\alpha([x,y])=[\alpha(x),\alpha(y)];$
\item regular if $\alpha$ is an automorphism;
\item involutive  if $\alpha$ is an involution, that is $\alpha^{2}=id .$
\end{itemize}
The center of the Hom-Lie superalgebra,  denoted $\mathcal{Z} (\mathcal{G})$, is defined by $$\mathcal{Z} (\mathcal{G})=\{x\in \mathcal{G}:\ [x,y]=0\  \forall y\in \mathcal{G} \}.$$
\end{defn}

The following theorem generalizes the twisting principle stated  in \cite{AmmarMakhloufJA2010,Yau:homology} in the following sense: starting from a Hom-Lie superalgebra and an even
Lie superalgebras endomorphism, we construct a new Hom-Lie superalgebra.

\begin{thm}
Let $(\mathcal{G},[\cdot,\cdot],\alpha)$ be a Hom-Lie superalgebra,  and  $\beta:\mathcal{G}\rightarrow \mathcal{G} $  be an
  even Lie superalgebra endomorphism.
Then
$(\mathcal{G},[\cdot,\cdot]_\beta,\beta o \alpha  )$, where $[x,y]_\beta=\beta([x,y])$,  is a Hom-Lie
superalgebra.

Moreover, suppose that  $(\mathcal{G}',[\cdot,\cdot]')$ is a Lie superalgebra and  $\alpha ' : \mathcal{G}'\rightarrow \mathcal{G}'$ is a Lie superalgebra endomorphism. If $f:\mathcal{G}\rightarrow \mathcal{G}'$ is a Lie superalgebra morphism that satisfies $f\circ \beta =\alpha'\circ f$ then
$$f:(\mathcal{G},[\cdot,\cdot]_\beta,\beta o \alpha  )\longrightarrow (\mathcal{G}',[\cdot,\cdot]',\alpha ')
$$
is a morphism of Hom-Lie superalgebras.
\end{thm}

\begin{proof}
We show that $(\mathcal{G},[\cdot,\cdot]_\beta,\beta o \alpha)$ satisfies the graded Hom-Jacobi identity(\ref{jacobie}). Indeed
\begin{eqnarray*}
\circlearrowleft_{x,y,z}(-1)^{|x||z|}\Big[\beta o  \alpha (x),[y,z]_{\beta}\Big]_{\beta}&=&\circlearrowleft_{x,y,z}(-1)^{|x||z|}\Big[\beta o  \alpha (x),\beta([y,z])\Big]_{\beta}\\
&=&\beta^{2} \Bigg(\circlearrowleft_{x,y,z}(-1)^{|x||z|}\Big[  \alpha (x),[y,z]\Big]\Bigg)\\
&=&0.
\end{eqnarray*}

The second assertion follows from
$$ f\Big( [x,y]_\beta \Big)= f\Big( [\beta(x),\beta(y)]\Big)
=\Big[f\circ \beta (x),f\circ \beta (y)\Big]'
= \Big[\alpha'\circ f(x),\alpha' \circ f (y)\Big]'=\Big[f(x),f(y)\Big]'_{\alpha'}. $$
\end{proof}
\begin{example}
We derive the following particular cases:
\begin{enumerate}
\item If $(\mathcal{G},[\cdot,\cdot],\alpha)$ is  a multiplicative Hom-Lie superalgebra then, for any $n\in \mathbb{N}$, $ (\mathcal{G},\alpha^n\circ [\cdot,\cdot],\alpha^{n+1})$ is  a multiplicative Hom-Lie superalgebra.
\item If $(\mathcal{G},[\cdot,\cdot])$ is  a Lie superalgebra and a self-map $\alpha$ on $\mathcal{G}$ is an even Lie superalgebra morphism then $(\mathcal{G},[\cdot,\cdot],\alpha)$ is  a multiplicative Hom-Lie superalgebra. 
\item If $(\mathcal{G},[.,.],\alpha)$ be a regular Hom-Lie superalgebra, then $(\mathcal{G},{\alpha^{-1}}\circ [.,.])$ is a Lie superalgebra.
\end{enumerate}
\end{example}


%

In the following we construct Hom-Lie superalgebras involving elements of the centroid of Lie superalgebras. Let $(\mathcal{G},[.,.])$ be a Lie superalgebra.
The centroid is defined by
$$Cent(\mathcal{G})=\{\theta \in End(\mathcal{G}):\theta([x,y])=[\theta(x),y], \ \forall x,y \in \mathcal{G}\}=(Cent(\mathcal{G}))_{0}\oplus (Cent(\mathcal{G}))_{1}.$$ The $Cent(\mathcal{G})$ is a subsuperpace of $End(\mathcal{G}).$
\begin{prop}
Let $(\mathcal{G},[.,.])$ be a Lie superalgebra and $\theta \in (Cent(\mathcal{G}))_{0}\subset(End(\mathcal{G}))_{0}.$ Set for $x,y \in \mathcal{G}$ $$\{x,y\}=\theta([x,y]).$$
Then $(\mathcal{G},\{.\},\theta)$  is a Hom-Lie superalgebras.
\end{prop}

\begin{proof}
For $\theta \in (Cent(\mathcal{G}))_{0},$ we have \\
$$\{x,y\}=\theta([x,y])=-(-1)^{|x||y|}\theta([y,x])=-(-1)^{|x||y|}[\theta(y),x]=[x,\theta(y)].$$
Then $\{x,y\}=[x,\theta(y)]=(-1)^{|x||y|}[\theta(y),x] =- (-1)^{|x||y|}   \{y,x\}$.\\
Also we have
\begin{eqnarray*}
  \{\theta(x),\{y,z\}\} =    \{\theta(x),[y,\theta(z)]\}
= \Big[\theta(x),\theta([y,\theta(z)])\Big] .                                                                  = \Big[\theta(x),[\theta(y),\theta(z)])\Big]
\end{eqnarray*}
It follows that $$\circlearrowleft_{x,y,z}(-1)^{|x||z|}\{\theta(x),\{y,z\}\}=\circlearrowleft_{x,y,z}(-1)^{|\theta(x)||\theta(z)|}\Big[\theta(x),[\theta(y),\theta(z)]\Big]
=0.$$
Since $(\mathcal{G},[.,.])$ is a Lie superalgebra, then  the super Hom-Jacobi identity is satisfied. Thus $(\mathcal{G},\{. , .\},\theta)$  is a Hom-Lie superalgebra.
\end{proof}

\section{Derivations of  Hom-Lie superalgebras}
We provide in the following a graded version of the study of derivations of Hom-Lie algebras stated in \cite{sheng}. Let $(\mathcal{G},[.,.],\alpha)$ be a Hom-Lie superalgebra, denote by   $\alpha^{k}$  the $k$-times composition of $\alpha,$ i.e.
$\alpha^{k}=\alpha\circ \cdots \circ \alpha (\textrm{$k$-times}).$
In particular, $\alpha^{-1}=0,\ \alpha^{0}=Id $ and $\alpha^{1}=\alpha.$\\
\begin{defn}
For any $k\geq -1,$ we call $D\in (End(\mathcal{G}))_{i}$ where  $i \in\mathbb{Z}_{2}$ a $\alpha^{k}$- derivation of the Hom-Lie superalgebra $(\mathcal{G},[.,.],\alpha)$ if $\alpha \circ D=D \circ \alpha$ and
\begin{eqnarray*}
D([x,y])=[D(x),  \alpha^{k}(y)]+(-1)^{|x||D|}[\alpha^{k}(x),D(y)]\ \text{ for  all  homogeneous elements }  x, y \in \mathcal{G}.
\end{eqnarray*}
\end{defn}

We denote by $Der_{\alpha^{k}}(\mathcal{G})=(Der_{\alpha^{k}}(\mathcal{G}))_{0}\oplus (Der_{\alpha^{k}}(\mathcal{G})_{1}$ the set of $\alpha^{k}$-derivations of the Hom-Lie superalgebra $(\mathcal{G},[.,.],\alpha),$
and  $$\displaystyle Der(\mathcal{G})=\oplus_{k\geq-1}Der_{\alpha^{k}}(\mathcal{G}).$$

 For any  homogeneous element $a \in \mathcal{G} ,$ satisfying $\alpha(a)=a,$ define $ad_{k}(a)\in End(\mathcal{G})$ by $$ad_{k}(a)(x)=[a,\alpha^{k}(x)],\ \forall x\in\mathcal{G}.$$ 
Notice that $ad_{k}(a)$ and $a$ are of the same parity.
\begin{prop}Let $(\mathcal{G},[.,.],\alpha)$ be a multiplicative Hom-Lie superalgebra.
Then $ad_{k}(a)$ is an $\alpha^{k+1}$-derivation, which we call inner $\alpha^{k+1}$-derivation.
 \end{prop}

 \begin{proof}
 Indeed we have $$ ad_{k}(a) \circ \alpha(x)=[a,\alpha^{k+1}(x)]= [\alpha(a),\alpha^{k+1}(x)]=\alpha\Big([a,\alpha^{k}(x)]\Big)=\alpha \circ  ad_{k}(a)(x)            $$
and

\begin{eqnarray*}
  &&ad_{k}(a)\Big([x,y]\Big) =[a,\alpha^{k}([x,y])]\\
 &=&[\alpha(a),[\alpha^{k}(x),\alpha^{k}(y)]] \\
  &=&-(-1)^{|a||y|} \Bigg((-1)^{|x||a|}\Big[\alpha ^{k+1}(x),[\alpha^{k}(y),a]\Big]
  +(-1)^{|y||x|}\Big[\alpha^{k+1}(y)),[a,\alpha^{k}(x)]\Big]\Bigg)\\
  &=&(-1)^{|a||y|} \Bigg((-1)^{|x||a|}  (-1)^{|y||a|}                       \Big[\alpha^{k+1} (x)),[a,\alpha^{k}(y)]\Big]
  +(-1)^{|y||x|}    (-1)^{|y||[a,x]|}       \Big[[a,\alpha^{k}(x)],\alpha^{k+1}(y))\Big]\Bigg)\\
 &=&\Big[[a,\alpha^{k}(x)],\alpha ^{k+1} (y)\Big]+ (-1)^{|x||a|} \Big[ \alpha ^{k+1}(x),[a,\alpha^{k+1}(y)]   \Big] \\
  &=&[ ad_{k}(a)(x), \alpha^{k+1}  (y)\Big]+ (-1)^{|x||a|} [ \alpha^{k+1} (x), ad_{k}(a)(y)].
\end{eqnarray*}
Therefore, $ad_{k}$ is a $\alpha ^{k+1}$-derivation. We denote by $Inn_{\alpha^{k}}(\mathcal{G})$ the set of inner $\alpha^{k}$-derivations, i.e.
$$Inn_{\alpha^{k}}(\mathcal{G})=\{[a,\alpha^{k-1}(.)]/ a\in \mathcal{G}_0\cup\mathcal{G}_1, \ \alpha(a)=a\}.$$ \\
For any $D\in Der(\mathcal{G})$ and $D'\in Der(\mathcal{G}),$ define their commutator $[D,D']$ as usual:
\begin{eqnarray}
 [D,D']  &=&D\circ D'-(-1)^{|D||D'|}D'\circ D.\label{der2}
\end{eqnarray}
\end{proof}
\begin{lem}\label{der1}
For any $D\in (Der_{\alpha^{k}}(\mathcal{G}))_{i}$ and $D'\in( Der_{\alpha^{s}}(\mathcal{G}))_{j},$ where $k+s\geq -1$ and $(i,j)\in \mathbb{Z}_2^2$, we have
$$[D,D']\in (Der_{\alpha^{k+s}}(\mathcal{G}))_{|D|+|D'|}.$$
\end{lem}
\begin{proof}
For any $x,y \in \mathcal{G}, $ we have
\begin{eqnarray*}
\Big[D,D'\Big]([x,y])&=&D\circ D'\Big([x,y] \Big)-(-1)^{|D||D'|}D' \circ D\Big([x,y] \Big)\\
&=&D\Big([D'(x),\alpha^{s}(y)]+(-1)^{|x||D'|}[\alpha^{s}(x),D'(y)]\Big)\\
&&-\ (-1)^{|D||D'|}D'\Big([D(x),\alpha^{k}(y)]+(-1)^{|x||D|}[\alpha^{k}(x),D(y)]\Big)\\
&=&[DD'(x),\alpha^{k+s}(y)]+(-1)^{|D||D'(x)|}[\alpha^{k} D'(x),D\alpha^{s}(y)]\\
&&+(-1)^{|x||D'|}\Big([D\alpha^{s}(x),\alpha^{k}  D'(y)]+(-1)^{|x||D|}[\alpha^{k+s}(x),DD'(y)]\Big)\\
&&-(-1)^{|D||D'|}\Big([D'D(x), \alpha^{k+s}(y)]+(-1)^{|D'||D(x)|}[\alpha^{s} D(x),D'\alpha^{k}(y)]\Big)\\
&&-(-1)^{|D||D'|}(-1)^{|x||D|}\Big([D' \alpha^{k}(x), \alpha^s D(y)]+(-1)^{|x||D'|}[ \alpha^{k+s}(x),D'D(y)]\Big).
\end{eqnarray*}
Since $D$ and  $D'$ \   satisfy $D\circ\alpha=\alpha \circ D$ and $D'\circ\alpha=\alpha \circ D',$\  we have
\begin{eqnarray*}
\Big[D,D'\Big]([x,y])
&=&\Big[DD'(x)-(-1)^{|D||D'|}D'D(x),\alpha^{k+s}(y)\Big]\\
&&+ \ (-1)^{(|x||D'|}(-1)^{(|x||D|}\Big[ \alpha^{k+s}(x),DD'(y)-(-1)^{|D||D'|}D'D(y)\Big]\\
&=&\Big[[D,D'](x),\alpha^{k+s}(y)\Big]
+(-1)^{|[D,D']||x|}\Big[\alpha^{k+s}(x),[D,D'](y)\Big].\\
\end{eqnarray*}
It is easy to verify that $\alpha \circ [D,D']=[D,D']\circ\alpha.$ 
Which leads to  $[D,D']\in Der_{\alpha^{k+s}}(\mathcal{G}).$
\end{proof}
\begin{rem}
Obviously, we have
$$Der_{\alpha^{-1}}=\{D \in End(\mathcal{G}):\  D\circ\alpha =\alpha \circ D, D([x,y])=0,\ \forall x,y\ \in \mathcal{G}\}.$$
Thus for any $D,\ D' \in Der_{\alpha^{-1}}(\mathcal{G}),$ we have $[D,D']\in Der_{\alpha^{-1}}(\mathcal{G}).$
\end{rem}

By Lemma \ref{der1}, obviously we have
\begin{prop}
With the above notations, $Der (\mathcal{G})$ is a Lie superalgebra, in which the bracket is given by \eqref{der2}.
\end{prop}
\begin{prop}
If we consider on $Der(\mathcal{G})$ the endomorphism $\widetilde{\alpha}$ defined by $\widetilde{\alpha}(D)=\alpha\circ D,$ then $(Der({\mathcal{G}}),[.,.],\widetilde{\alpha})$ is a  Hom-Lie superalgebra where  $[.,.]$ is given by \eqref{der2}.
\end{prop}
Now, we consider extensions of a  Hom-Lie superalgebra $(\mathcal{G},[.,.],\alpha)$ using derivations.
 For any $D\in (End(\mathcal{G}))_{i},$ consider the vector space   $\widetilde{\mathcal{G}_{0}}=\mathcal{G}_{0}\oplus \mathbb{R}D,$  $\widetilde{\mathcal{G}_{1}}=\mathcal{G}_{1}$  and $\widetilde{\mathcal{G}}=
\widetilde{\mathcal{G}_{0}}\oplus \widetilde{\mathcal{G}_{1}}.$ Define a skew-symmetric bilinear bracket operation $[.,.]_{D}$ on $\widetilde{\mathcal{G}}$ by
$$[g+\gamma D,h+\lambda D]_{D}=[g,h]-\lambda D(g) +\gamma D(h), \forall g,h \in \mathcal{G}.$$
Define $\alpha_{D}\in End(\mathcal{G}\oplus \mathbb{R}D)$ by $\alpha_{D}(g+\lambda D)=\alpha (g)+\lambda D.$
\begin{prop}
With the above  notations, $(\widetilde{\mathcal{G}},[.,.]_{D},\alpha_{D})$ is a  Hom-Lie superalgebra if and only if $D$ is a derivation of the Hom-Lie superalgebra $(\mathcal{G},[.,.],\alpha).$
\end{prop}
\section{ Representations and Cohomology of  Hom-Lie Superalgebras}
In this section we study representations of Hom-Lie Superalgebras and define a family of cohomologies by providing a family of coboundary operator defining cohomology complexes.\\

\subsection{Representations of Hom-Lie superalgebras}
Let $(\mathcal{G},[.,.],\alpha)$ be a Hom-Lie superalgebra and $V=V_{0}\oplus V_{1}$ be an arbitrary vector superspace. Let $\beta\in\mathcal{G}l(V)$ be an arbitrary even linear self-map on $V$  and
\begin{eqnarray*}
[.,.]_{V}&:&\mathcal{G}\times V \rightarrow V \\
&&(g,v)\mapsto [g,v]_{V}
\end{eqnarray*}
 a bilinear map  satisfying $[\mathcal{G}_{i},V_{j}]_{V}\subset V_{i+j}$ where $i,j\in \mathbb{Z}_{2}.$
\begin{defn}
The triple $(V,[.,.]_{V}, \beta)$  is called a Hom-module on the Hom-Lie superalgebra $\mathcal{G}=\mathcal{G}_{0}\oplus \mathcal{G}_{1}$ or  $\mathcal{G}$-Hom-module $V$  if the  even bilinear map $[.,.]_{V}$ satisfies

\begin{eqnarray}
[\alpha(x),\beta(v)]_{V}&=&\beta([x,v]_{V}) \label{rep1}
\end{eqnarray}
 and
\begin{eqnarray}
\left[[x,y],\beta(v)\right]_{V}&=&\left[\alpha(x),[y,v]\right]_{V}-(-1)^{|x||y|}\left[\alpha(y),[x,v]\right]_{V} \label{rep2}
\label{mod}
\end{eqnarray}
 for all homogeneous elements $x, y \in \mathcal{G}$ and $v\in  V .$\\
Hence, we  say that $(V,[.,.]_{V}, \beta)$ is a representation of $\mathcal{G}.$
\end{defn}
\begin{example}
Let $(\mathcal{G},[.,.],\alpha)$ be a Hom-Lie superalgebra  and  $ad:\mathcal{G} \rightarrow End(\mathcal{G})$ be an operator defined for $x\in\mathcal{G} $ by $ad(x)(y)=[x,y].$ Then $(\mathcal{G},ad,\alpha)$ is a representation of $\mathcal{G}$.
\end{example}
\begin{example}
Given a representation $(V,[.,.]_{V},\beta)$ of a Hom-Lie superalgebra $(\mathcal{G},[.,.],\alpha).$\\ Denote   $\widetilde{\mathcal{G}}=\mathcal{G}\oplus V$ and       $\widetilde{\mathcal{G}}_{k}=\mathcal{G}_{k}\oplus V_k.$ If $x\in \mathcal{G}_{i}$ and $v\in V_{i} \ (i\in \mathbb{Z}_{2})$, we denote $|(x,v)|=|x|.$\\
 Define a super skew-symmetric  bracket  $[.,.]_{\widetilde{\mathcal{G}}}:\wedge ^{2}(\mathcal{G}\oplus V)\rightarrow \mathcal{G}\oplus V$ by
$$[(x,u),(y,v)]_{\widetilde{\mathcal{G}}}=\Big([x,y],\ [x,v]_{V}-(-1)^{|x||y|}[y,u]_{V}\Big).$$
Define $\widetilde{\alpha}:\mathcal{G}\oplus V\rightarrow \mathcal{G}\oplus V$ by $\widetilde{\alpha}(x,v)=(\alpha(x),\beta(v)).$\\  Then $(\mathcal{G}\oplus V,[.,.]_{\widetilde{\mathcal{G}}},\widetilde{\alpha})$ is Hom-Lie superalgebra, which we call semi-direct product of the Hom-Lie superalgebra $(\mathcal{G},[.,.],\alpha)$ by $V.$
\end{example}
\begin{rem}
When $\beta$ is the zero-map, we say that the module $V$ is trivial.
\end{rem}


\subsection{Cohomology of  Hom-Lie superalgebras}

Let   $x_1,\cdots ,x_k$ be $k$ homogeneous elements of $\mathcal{G},$ we denote by $|(x_1,\cdots ,x_k)|=|x_1|+\cdots +|x_k|$ the parity of an element $(x_{1},\dots,x_{k})$ in $\mathcal{G}^k$.\\
The set $C^{k}(\mathcal{G},V)$ of $k$-cochains on space $\mathcal{G}$ with values in $V,$ is the set of $k$-linear maps $f:\otimes^{k}\mathcal{G}\rightarrow V$
satisfying
$$f(x_{1},\dots,x_{i},x_{i+1},\dots,x_{k})=-(-1)^{|x_{i}||x_{i+1}|}f(x_{1},\dots,x_{i+1},x_{i},\dots,x_{k})\ \textrm{ for } 1\leq i\leq k-1 .$$
For $k=0$ we have $C^{0}(\mathcal{G},V)=V.$\\
The map $f$ is called even (resp. odd) when  $f(x_{1},\dots,x_{k})\in V_0$ (resp. $f(x_{1},\dots,x_{k})\in V_1$)  for all even (resp odd ) elements $(x_{1},\dots,x_{k})\in \mathcal{G}^k.$\\
A $k$-hom-cochain on $\mathcal{G}$ with values in  $V$ is defined to be a $k$-cochain $f \in C^{k}(\mathcal{G},\ V)$ such that it is compatible with $\alpha $ and $\beta$ in the sense that $\beta \circ f=f \circ \alpha,$ i.e.
$$\beta \circ f (x_{1},\dots ,x_{k})=f(\alpha(x_{1}),\dots ,\alpha(x_{k})).$$
Denote $C^{k}_{\alpha ,\beta}(\mathcal{G},\ V)$ the set of $k$-hom-cochains:
\begin{eqnarray}
C^{k}_{\alpha ,\beta}(\mathcal{G},\ V)&=&
\{f\in C^{k}(\mathcal{G},\ V) :\  \beta  \circ f=f \circ \alpha \}. \label{coh}
\end{eqnarray}
For a given $r$, we define a map  $\delta^{k}_r:C^{k}(\mathcal{G},\ V)\rightarrow C^{k+1}(\mathcal{G},\ V)$ by setting

\begin{eqnarray}
&&\delta^{k}_{r}(f)(x_{0},\dots,x_{k})\nonumber \\&&=\sum_{0\leq s < t\leq k}(-1)^{t+|x_{t}|(|x_{s+1}|+\dots+|x_{t-1}|)}
f\Big(\alpha(x_{0}),\dots,\alpha(x_{s-1}),[x_{s},x_{t}],\alpha(x_{s+1}),\dots,\widehat{x_{t}},\dots,\alpha(x_{k})\Big) \nonumber \\
&&+\sum_{s=0}^{k}(-1)^{s+|x_{s}|(|f|+|x_{0}|+\dots+|x_{s-1}|)}\Bigg[\alpha^{k+r-1}(x_{s}), f\Big(x_{0},\dots,\widehat{x_{s}},\dots,x_{k}\Big)\Bigg]_{V},\label{def cob}
\end{eqnarray}
 where $f\in C^{k}(\mathcal{G},\ V)$,  $|f|$ is the parity of $f$, $\ x_{0},....,x_{k}\in \mathcal{G}$ and $\ \widehat{x_{i}}\ $  means that $x_{i}$ is omitted.\\

In the sequel we assume that the Hom-Lie superalgebra $(\mathcal{G},[.,.],\alpha)$ is multiplicative.
\begin{lem}
With the above notations, for any $f \in C^{k}_{\alpha ,\beta}(\mathcal{G},\ V),$ we have $$\delta^{k}_{r}(f)\circ \alpha = \beta \circ \delta^{k}_{r}(f) .$$ Thus we obtain a well-defined map $$\delta_{r}^{k}:C^{k}_{\alpha ,\beta}(\mathcal{G},\ V)\rightarrow C^{k+1}_{\alpha ,\beta}(\mathcal{G},\ V).$$
\end{lem}
\begin{proof}
Let $f\in C^{k}_{\alpha ,\beta}(\mathcal{G},\ V)$ and $(x_{0},\cdots,x_{k}) \in \mathcal{G}^{k+1}.$\\
{\footnotesize
\begin{eqnarray*}
&&\delta^{k}_{r}(f)\circ \alpha (x_{0},\dots,x_{k})\\
&=&\delta^k(f)(\alpha (x_{0}),\dots,\alpha(x_{k}))\\
&=&\sum_{0\leq s < t\leq k}(-1)^{t+|x_{t}|(|f|+|x_{s+1}|+\dots+|x_{t-1}|)}
f\Big(\alpha^{2}(x_{0}),\dots,\alpha^{2}(x_{s-1}),[\alpha(x_{s}),\alpha(x_{t})],\alpha^{2}(x_{s+1}),
\dots,\widehat{x_{t}},\dots,\alpha^{2}(x_{k})\Big)\\
&&+\sum_{s=0}^{k}(-1)^{s+|x_{s}|(|f|+|x_{0}|+\dots+|x_{s-1}|)}\Bigg[\alpha^{k+r}(x_{s}), f\Big(\alpha(x_{0}),\dots,\widehat{x_{s}},\dots,\alpha(x_{k})\Big)\Bigg]_{V}\\
&=&\sum_{0\leq s < t\leq k}(-1)^{t+|x_{t}|(|f|+|x_{s+1}|+\dots+|x_{t-1}|)}
f\circ\alpha\Big(\alpha(x_{0}),\dots,\alpha(x_{s-1}),[x_{s},x_{t}],\alpha(x_{s+1}),
\dots,\widehat{x_{t}},\dots,\alpha(x_{k})\Big)\\
&&+\sum_{s=0}^{k}(-1)^{s+|x_{s}|(|f|+|x_{0}|+\dots+|x_{s-1}|)}\Bigg[\alpha^{k+r}(x_{s}), f\circ\alpha\Big(x_{0},\dots,\widehat{x_{s}},\dots,x_{k}\Big)\Bigg]_{V}\\
&=&\sum_{0\leq s < t\leq k}(-1)^{t+|x_{t}|(|f|+|x_{s+1}|+\dots+|x_{t-1}|)}
\beta \circ f\Big(\alpha(x_{0}),\dots,\alpha(x_{s-1}),[x_{s},x_{t}],\alpha(x_{s+1}),
\dots,\widehat{x_{t}},\dots,\alpha(x_{k})\Big)\\
&&+\sum_{s=0}^{k}(-1)^{s+|x_{s}|(|f|+|x_{0}|+\dots+|x_{s-1}|)} \Bigg[\alpha^{k+r}(x_{s}), \beta \circ f\Big(x_{0},\dots,\widehat{x_{s}},\dots,x_{k}\Big)\Bigg]_{V}\\
&=&\sum_{0\leq s < t\leq k}(-1)^{t+|x_{t}|(|f|+|x_{s+1}|+\dots+|x_{t-1}|)}
\beta \circ f\Big(\alpha(x_{0}),\dots,\alpha(x_{s-1}),[x_{s},x_{t}],\alpha(x_{s+1})
\dots,\widehat{x_{t}},\dots,\alpha(x_{k})\Big)\\
&&+\sum_{s=0}^{k}(-1)^{s+|x_{s}|(|f|+|x_{0}|+\dots+|x_{s-1}|)}\beta \Bigg( \Bigg[\alpha^{k+r-1}(x_{s});  f\Big(x_{0},\dots,\widehat{x_{s}},\dots,x_{k}\Big)\Bigg]_{V}\Bigg)\\
&=&\beta \circ \delta ^{k}_{r}(k)(x_{0},\dots,x_{k}),
\end{eqnarray*}}
which completes the proof.
\end{proof}

\begin{thm}
Let Let $(\mathcal{G},[.,.],\alpha)$ be a multiplicative Hom-Lie superalgebra and $(V,[.,.]_{V},\beta)$ be a $\mathcal{G}$-Hom-module.
 
For a given integer  $r\geq 1,$  the pair $(\bigoplus_{k>0}C^{k}_{\alpha ,\beta}(\mathcal{G},\ V), \{\delta^{k}_{r}\}_{k>0})$ 
  defines a cohomology complex, that is  $\delta^{k}_{r} \circ \delta_{r}^{k-1}=0.$
\end{thm}

\begin{proof}
For any $f\in C^{k-1}_{\alpha ,\beta}(\mathcal{G},\ V)   $  we have
\begin{eqnarray}
 && \delta^{k} _{r} \circ \delta^{k-1} _{r}(f)(x_{0},\dots,x_{k}) \nonumber \\
&&=\sum_{s<t}(-1)^{t+|x_{t}|(|x_{s+1}|+\dots+|x_{t-1}|)}\delta^{k-1}(f) (\alpha(x_{0}),\dots,\alpha(x_{s-1}),[x_{s},x_{t}],\alpha(x_{s+1}),\dots,\widehat{x_{t}},\dots,\alpha(x_{k}))\ \ \ \ \ \ \ \ \label{cob1}\\
&&+\sum_{s=0}^{k}(-1)^{s+|x_{s}|(|f|+|x_{0}|+\dots+|x_{s-1}|)}\Big[\alpha^{k+r-1}(x_{s}),\delta^{k-1}_{r}(f) (x_{0},\dots,\widehat{x_{s}},\dots,x_{k})\Big]_{V}. \label{cob2}
\end{eqnarray}
From (\ref{cob1}) we have
{\footnotesize
\begin{eqnarray}
&&\delta^{k-1} (f)\left( \alpha(x_{0}),\dots,\alpha(x_{s-1}),[x_{s},x_{t}],\alpha(x_{s+1}),\dots,\widehat{x_{t}},\dots,\alpha(x_{k})  \right)\nonumber \\
&=& \sum_{s'<t'<s}(-1)^{t'+|x_{t'}|(|x_{s'+1}|+\dots+|x_{t'-1}|)}f \Big(\alpha^{2}(x_{0}),\dots,\alpha^{2}(x_{s'-1}),[\alpha(x_{s'}),\alpha(x_{t'})],\alpha^{2}(x_{s'+1})
,\dots,\widehat{x_{t'}},\nonumber \\
 &&\dots,\alpha^{2}(x_{s-1}),\alpha([x_{s},x_{t}]),\alpha^{2}(x_{s+1}),\dots,\widehat{x_{t}},\dots,
\alpha^{2}(x_{k})\Big)+\label{cob3}\\
&&\sum_{s'<s}(-1)^{s+|x_{s}|(|x_{s'+1}|+\dots +|x_{s-1}|)}f\Big(\alpha^{2}(x_{0}),\dots,\alpha^{2}(x_{s'-1}),
\Big[\alpha(x_{s'-1}),[x_{s},x_{t}]\Big], \alpha^{2}(x_{s'+1})
,\dots,\widehat{x_{s,t}}, \dots,\alpha^{2}(x_{k})\Big)\ \ \ \label{cob4} \\
&&+\sum_{s'<s<t'<t}(-1)^{t'+|x_{t'}|(|x_{s'+1}|+\dots+|[x_{s},x_{t}]|+ \dots +|x_{t'-1}|)}f\Big(\alpha^{2}(x_{0}),\dots,\alpha^{2}(x_{s'-1}),[\alpha(x_{s'}),\alpha(x_{t'})],\alpha^{2}(x_{s'+1})
,\nonumber \\
&& \dots,\alpha([x_{s},x_{t}]),\dots,\widehat{x_{t'}}, \dots,\alpha^{2}(x_{k})\Big)\label{cob5}\\
&&+\sum_{s'<s<t<t'}(-1)^{t'+|x_{t'}|(|x_{s'+1}|+\dots+|x_{s-1}|+|[x_{s},x_{t}]|+|x_{s+1}|+\dots+\widehat{|x_{t}|}+\dots+
|x_{t'-1}|)}f\Big(\alpha^{2}(x_{0}),\dots,\alpha^{2}(x_{s'-1}),
\nonumber\\
&&
[\alpha(x_{s'}),\alpha(x_{t'})],\alpha^{2}(x_{s'+1}),\dots,
 \dots,\alpha([x_{s},x_{t}]),\dots,\widehat{x_{t}}, \dots,\widehat{x_{t'}},\dots,\alpha^{2}(x_{k})\Big)\label{cob6}\\
&&+\sum_{s<t'<t}(-1)^{t'+|x_{t'}|(|x_{s+1}|+\dots+|x_{t'-1}|)}f\Big(\alpha^{2}(x_{0}),\dots,\Big[[x_{s},x_{t}],\alpha
(x_{t'})\Big],\alpha^{2}(x_{s+1}),\dots,\widehat{x_{t,t'}},\dots,\alpha^{2}(x_{k})\Big) \label{cob7}\\
&&+\sum_{s<t<t'}(-1)^{t'-1+|x_{t'}|(|x_{s+1}|+\dots+\widehat{|x_{t}|}+\dots+|x_{t'-1}|)}f\Big(\alpha^{2}(x_{0}),\dots,
\alpha^{2}(x_{s-1}),\Big[[x_{s},x_{t}],\alpha
(x_{t'})\Big],\alpha^{2}(x_{s+1}),\nonumber \\
&&\dots,\widehat{x_{t,t'}},\dots,\alpha^{2}(x_{k})\Big)\label{cob8}\\
&&+\sum_{s<s'<t'<t}(-1)^{t'+|x_{t'}|(|x_{s'+1}|+\dots+|x_{t'-1}|)}f\Big(\alpha^{2}(x_{0}),\dots,
\alpha^{2}(x_{s-1}),\alpha([x_{s},x_{t}]),\nonumber \\
&& \dots,[\alpha(x_{s'}),\alpha(x_{t'})],\dots,\widehat{x_{t'}},\dots,\widehat{x_{t}},\dots,\alpha^{2}
(x_{k})
\Big)\label{cob9}\\
&&+\sum_{s<s'<t<t'}(-1)^{t'-1+|x_{t'}|(|x_{s'+1}|+\dots+\widehat{x_{t}}+\dots+|x_{t'-1}|)}f\Big(\alpha^{2}(x_{0}),\dots,
\alpha^{2}(x_{s-1}),\alpha([x_{s},x_{t}]),\alpha^{2}(x_{s+1}),\nonumber \\
&&\dots,[\alpha(x_{s'}),\alpha(x_{t'})]\dots,\widehat{x_{t}},\dots, \widehat{x_{t'}},\alpha^{2}(x_{k})\Big)\label{cob10}\\
&&+\sum_{t<s'<t'}(-1)^{t'+|x_{t'}|(|x_{s'+1}|+\dots+\widehat{x_{t,t'}}+\dots+|x_{t'-1}|)}f\Big(\alpha^{2}(x_{0}),\dots,
\alpha^{2}(x_{s-1}),\alpha([x_{s},x_{t}]),\alpha^{2}(x_{s+1}),\nonumber \\
&&\dots,\widehat{x_{t}} \dots,[\alpha(x_{s'}),\alpha(x_{t'})],\dots\widehat{x_{t'}},\dots,\alpha^{2}(x_{k})\Big)\label{cob11}\\
&&+\sum_{0\leq s' <s}(-1)^{s'+|x_{s'}|(|f|+|x_{0}|+\dots+|x_{s'-1}|)}\Bigg[\alpha^{k+r-1}(x_{s'}),f\Big(\alpha(x_{0})
\dots,\widehat{x_{s'}},\dots,[x_{s},x_{t}],\dots,\widehat{x_{t}},\dots,\alpha(x_{k})\Big)\Bigg]_{V}\label{cob12}\\
&&+(-1)^{s+|[x_{s},x_{s}]|(|f|+|x_{0}|+\dots+|x_{s-1}|)}\Bigg[\alpha^{k+r-2}([x_{s},x_{t}]),f\Big(\alpha(x_{0})
\dots,\widehat{[x_{s},x_{t}]},\alpha(x_{s+1}),\dots,\widehat{x_{t}},\dots,\alpha(x_{k})\Big)\Bigg]_{V}\label{cob13}
\end{eqnarray}}
{\footnotesize
\begin{eqnarray}
&&+\sum_{ s <s'<t}(-1)^{s'+|x_{s'}|(|f|+|x_{0}|+\dots+|[x_{s},x_{t}]|+\dots+|x_{s'-1}|)}\Bigg[\alpha^{k+r-1}(x_{s'}),
f\Big(\alpha(x_{0})
\dots,[x_{s},x_{t}],\dots,\widehat{x_{s',t}},\dots,\alpha(x_{k})\Big)\Bigg]_{V}\ \label{cob14} \\
&&+\sum_{ t <s'}(-1)^{s'+|x_{s'}|(|f|+|x_{0}|+\dots+|[x_{s},x_{t}]|+\dots+\widehat{|x_{t}|}+\dots,|x_{s'-1}|)}
\Bigg[\alpha^{k+r-1}(x_{s'}),
f\Big(\alpha(x_{0})
\dots,[x_{s},x_{t}],\dots,\widehat{x_{t,s'}},\dots,\alpha(x_{k})\Big)\Bigg]_{V}.\label{cob15} \
\end{eqnarray}}
The identity \eqref{cob2}  implies that
\begin{eqnarray}
&&\Bigg[\alpha^{k+r-1}(x_{s}),\delta^{k-1}(f) (x_{0},\dots,\widehat{x_{s}},\dots,x_{k})\Bigg]_{V}\nonumber\\
&=&\Bigg[\alpha^{k+r-1}(x_{s}),\ \nonumber \\
&&\sum_{s'<t'<s}(-1)^{t'+|x_{t'}|(|x_{s'+1}|+\dots +|x_{t'-1}|)}f \Big(\alpha(x_{0}),\dots,\alpha(x_{s'-1}),[x_{s'},x_{t'}],\alpha(x_{s'+1}),\nonumber \\
&&\dots,\widehat{x_{s',t',s}},\alpha(x_{s+1}),\dots,\alpha(x_{k})\Big)\Bigg]_{V}\label{cob18}\\
&&+\Bigg[\alpha^{k+r-1}(x_{s}), \sum_{s'<s<t'}(-1)^{t'-1+|x_{t'}|(|x_{s'+1}|+\dots +\widehat{|x_{s}|}+\dots+|x_{t'-1}|)}f\Big(\alpha(x_{0}),\dots,\alpha(x_{s'-1}),[x_{s'},x_{t'}],\alpha(x_{s'+1}),\nonumber \ \ \\
&&\dots,
\widehat{x_{t,s'}},\dots,\alpha(x_{k})\Big)\Bigg]_{V}\label{cob19} \\
&&+\Bigg[\alpha^{k+r-1}(x_{s}), \sum_{s<s'<t'}(-1)^{t'+|x_{t'}|(|x_{s'+1}|+\dots +\dots+|x_{t'-1}|)}f\Big(\alpha(x_{0}),\dots,\widehat{x_{s}},\dots,\alpha(x_{s'-1}),[x_{s'},x_{t'}],\alpha(x_{s'+1})
,\nonumber \ \ \\
&&\dots,
\widehat{x_{t'}},\dots,\alpha(x_{k})\Big)\Bigg]_{V} \label{cob20}\\
&&+\Bigg[\alpha^{k+r-1}(x_{s}),\sum_{s'=0}^{s-1}(-1)^{s'+|x_{s'}|(|c|+|x_{0}|+\dots +|x_{s'-1}|)}\Big[\alpha^{k+r-2}(s'),f\Big(x_{0},\dots,\widehat{x_{s',s}},\dots,x_{k}\Big)\Big]_{V}\label{cob21}\\
&&+\Bigg[\alpha^{k+r-1}(x_{s}),\sum_{s'=s+1}^{k}(-1)^{s'-1+|x_{s'}|(|f|+|x_{0}|+\dots+\widehat{|x_{s}|},\dots, +|x_{s'-1}|)}\Big[\alpha^{k+r-2}(s'),f\Big(x_{0},\dots,\widehat{x_{s',s}},\dots,x_{k}\Big)\Big]_{V}\Bigg]_{V}\nonumber\\
\label{cob22}
\end{eqnarray}

Super-Hom-Jacobi identity leads to  $$\sum_{s<t }(-1)^{t+|x_{t}|(|x_{s+1}|+\dots+|x_{t-1}|)}\Big((\ref{cob4})+(\ref{cob7})+(\ref{cob8})\Big)=0.$$


Using (\ref{mod}) and (\ref{coh}), we obtain
{\footnotesize
\begin{eqnarray}
&&(\ref{cob13})
=\Bigg[\alpha^{k+r-2}([x_{s},x_{t}]);\  f\left( \alpha(x_{0}),\dots,\alpha(x_{s-1}),\widehat{\alpha([x_{s},x_{t}]},\alpha(x_{s+1}),\dots,
\widehat{x_{t}},\dots,
\alpha (x_{k})\right) \Bigg]_{V}\nonumber \\
&=&\Bigg[\alpha^{k+r-1}(x_{s}),\Big[\alpha^{k+r-2}(x_{t}), f\Big(x_{0},\dots,\widehat{x_{s,t}},\dots,
 x_{k}\Big)\Big]_{V}\Bigg]_{V}
-\Bigg[\alpha^{k+r-1}(x_{t}),\Big[\alpha^{k+r-2}(x_{s}), f\Big(x_{0},\dots,\widehat{x_{s,t}},\dots,
 x_{k}\Big)\Big]_{V}\Bigg]_{V} \ \ \ \label{cob24}.
\end{eqnarray}}

Thus
 \begin{footnotesize}
\begin{eqnarray*}
&&\sum_{s<t}(-1)^{t+|x_{t}|(|x_{s+1}|+\dots+|x_{t-1}|)}(\ref{cob13}) +\sum_{s=0}^{k}(-1)^{s+|x_{s}|(|f|+|x_{0}|+\dots+|x_{s-1}|)}(\ref{cob21})
+\sum_{s=0}^{k}(-1)^{s+|x_{s}|(|f|+|x_{0}|+\dots+|x_{s-1}|)}
(\ref{cob22})
 =0.
\end{eqnarray*}
\end{footnotesize}

By a simple calculation we get
\begin{eqnarray*}
&&\sum_{s<t}(-1)^{t+|x_{t}|(|x_{s+1}|+\dots+|x_{t-1}|)}(\ref{cob12})
+ \sum_{s=0}^{k}(-1)^{s+|x_{s}|(|f|+|x_{0}|+\dots+|x_{s-1}|)}(\ref{cob20})
=0,
\end{eqnarray*}
\begin{eqnarray*}
&&\sum_{s<t}(-1)^{t+|x_{t}|(|x_{s+1}|+\dots+|x_{t-1}|)}(\ref{cob14})
+\sum_{s=0}^{k}(-1)^{s+|x_{s}|(|f|+|x_{0}|+\dots+|x_{s-1}|)}(\ref{cob19})
=0,
\end{eqnarray*}
\begin{eqnarray*}
&&\sum_{s<t}(-1)^{t+|x_{t}|(|x_{s+1}|+\dots+|x_{t-1}|)}
(\ref{cob15})
+\sum_{s=0}^{k}(-1)^{s+|x_{s}|(|f|+|x_{0}|+\dots+|x_{s-1}|)}(\ref{cob18})
=0,
\end{eqnarray*}
and
\begin{eqnarray*}
&&\sum_{s<t}(-1)^{t+|x_{t}|(|x_{s+1}|+\dots+|x_{t-1}|)}\Big((\ref{cob5})+(\ref{cob10})\Big)\\
&=&\sum_{s<t}(-1)^{t+|x_{t}|(|x_{s+1}|+\dots+|x_{t-1}|)}\sum_{s'<s<t'<t}(-1)^{t'+|x_{t'}|(|x_{s'+1}|
+\dots+|[x_{s},x_{t}]|
+\dots +|x_{t'-1}|)}\\
&&f\Big(\alpha^{2}(x_{0}),\dots,\alpha^{2}(x_{s'-1}),[\alpha(x_{s'}),\alpha(x_{t'})],\alpha^{2}(x_{s'+1})
, \dots,\alpha([x_{s},x_{t}]),\dots,\widehat{x_{t'}}, \dots,\alpha^{2}(x_{k})\Big)\\
&+&\sum_{s<t}(-1)^{t+|x_{t}|(|x_{s+1}|+\dots+|x_{t-1}|)}\sum_{s<s'<t<t'}(-1)^{t'-1+|x_{t'}|(|x_{s'+1}|
+\dots+\widehat{x_{t}}
+\dots+|x_{t'-1}|)}\\
&&f\Big(\alpha^{2}(x_{0}),\dots,
\alpha^{2}(x_{s-1}),\alpha([x_{s},x_{t}]),\alpha^{2}(x_{s+1}),
\dots,\widehat{x_{t,t'}},\dots,[\alpha(x_{s'}),\alpha(x_{t'})]\dots,\alpha^{2}(x_{k})\Big).\\
&=&0.
\end{eqnarray*}
Similarly  $\displaystyle\sum_{s<t}(-1)^{t+|x_{t}|(|x_{s+1}|+\dots+|x_{t-1}|)}\Big((\ref{cob3})+(\ref{cob11})\Big)=0 $ and
$$\displaystyle\sum_{s<t}(-1)^{t+|x_{t}|(|x_{s+1}|+\dots+|x_{t-1}|)}\Big((\ref{cob6})+(\ref{cob9})\Big)=0 $$\\
Therefore $\delta^{k}_{r}\circ \delta^{k-1}_{r} =0.$ 
\end{proof}
The previous Theorem shows that we may have infinitely many cohomology complexes.

The corresponding cocycles, coboundaries and cohomology groups are defined as follows. 
\begin{defn}Let $(\mathcal{G},[.,.],\alpha)$ be a Hom-Lie superalgebra and  $(V,[.,.]_{V}, \beta)$ be a Hom-module. We have with respect the $r$-cohomology defined by the coboundary operators
$$\delta_{r}^{k}:C^{k}_{\alpha ,\beta}(\mathcal{G},\ V)\rightarrow C^{k+1}_{\alpha ,\beta}(\mathcal{G},\ V),$$

\begin{itemize}
\item The  $k$-cocycles space is defined as $Z_{r}^{k}(\mathcal{G},V)=\ker \ \delta^{k}_{r}$. The  even (resp. odd)  $k$-cocycles space  is defined as  
$Z_{r,0}^{k}(\mathcal{G},V)=Z_{r}^{k}(\mathcal{G},V)\cap (C^{k}_{\alpha,\beta}(\mathcal{G},\ V))_0$ (resp. $Z_{r,1}^{k}(\mathcal{G},V)=Z_{r}^{k}(\mathcal{G},V)
\cap (C^{k}_{\alpha,\beta}(\mathcal{G},\ V))_1$).
 \item The $k$-coboundaries space is defined as  $B_{r}^{k}(\mathcal{G},V)=Im \ \delta^{k-1}_{r}$. The even (resp. odd)  $k$-coboundaries space is $B_{r,0}^{k}(\mathcal{G},V)=B_{r}^{k}(\mathcal{G},V)\cap (C^{k}_{\alpha,\beta}(\mathcal{G},\ V))_0$ (resp. $B_{r,1}^{k}(\mathcal{G},V)=B_{r}^{k}(\mathcal{G},V)\cap (C^{k}_{\alpha,\beta}(\mathcal{G},\ V))_0$).
\item The $k^{th}$  cohomology space is the quotient $H_{r}^{k}(\mathcal{G},V)= Z_{r}^{k}(\mathcal{G},V)/ B_{r}^{k}(\mathcal{G},V)$. It decomposes as well as   even and odd  $k^{th}$  cohomology spaces.
  \end{itemize}
 Finally, we denote by $H^{k}_{r}(\mathcal{G},V)=H_{r,0}^{k}(\mathcal{G},V) \oplus H_{r,1}^{k}(\mathcal{G},V)$ the $k^{th}$ $r$-cohomology space and by $\oplus_{k\geq 0}H^{k}_{r}(\mathcal{G},V)$ the $r$-cohomology group of the Hom-Lie superalgebra $\mathcal{G}$ with values in $V$.
 \end{defn}
\begin{rem}
The $Z_{r}^{1}(\mathcal{G},\mathcal{G})$ is the set of $\alpha^{r}$-derivation of $\mathcal{G}$.
\end{rem}
\begin{example}
In this example we compute the second scalar cohomology group of the Hom-Lie superalgebra  $osp(1,2)_{\lambda}$ constructed in  \cite{AmmarMakhloufJA2010}.

Let $osp(1,2)=A_0\oplus A_1$ be the vector  superspace where $V_0$ is generated by \\
\begin{center}
$H=\left(\begin{array}{ccc}1&0&0\\
0&0&0\\
0&0&-1
\end{array}\right) ,$ \
$X=\left(\begin{array}{ccc}0&0&1\\
0&0&0\\
0&0&0
\end{array}\right), $
$Y=\left(\begin{array}{ccc}0&0&0\\
0&0&0\\
1&0&0
\end{array}\right) ,$
\end{center}
and $V_1$ is generated by:
\begin{center}
$F=\left(\begin{array}{ccc}0&0&0\\
1&0&0\\
0&1&0
\end{array}\right) ,$
$G=\left(\begin{array}{ccc}0&1&0\\
0&0&-1\\
0&0&0
\end{array}\right) .$\end{center}
Let $\lambda \in \mathbb{R}^*, $ we consider the linear map $\alpha_\lambda:osp(1,2)\rightarrow osp(1,2)$ defined by:
$$\alpha_\lambda(X)=\lambda^2X, \ \alpha_\lambda(Y)=\frac{1}{\lambda^2}Y, \ \alpha_\lambda(H)=H,\ \alpha_\lambda(F)=\frac{1}{\lambda}F, \  \alpha_\lambda(G)=\lambda G.$$
We define a superalgebra bracket $[.,.]_\lambda$ with respect to the basis, for $\lambda\neq 0, $ by:
\begin{eqnarray*}
& [H,X]_\lambda=2\lambda^2X,\ [H,Y]_\lambda=-\frac{2}{\lambda^2}Y,\ [X,Y]_\lambda =H,\ [Y,G]_\lambda=\frac{1}{\lambda} F,\  [X,F]_\lambda=\lambda G,  [H,F]_\lambda=-\frac{1}{\lambda}F, 
\\ \  & [H,G]_\lambda=\lambda G,\ [G,F]_\lambda =H,   [G,X]=0, [Y,F]=0, [G,G]_\lambda=-2\lambda^2X,\ [F,F]_\lambda= \frac{2}{\lambda^2}Y.
\end{eqnarray*}
 Then, we have $osp(1,2)_{\lambda}=(osp(1,2),[.,.]_{\lambda},\alpha_\lambda)$ is a Hom-Lie superalgebra.\\

Let  $f\in C_{\alpha,Id_{\mathbb{C}}}^{1}(osp(1,2),\mathbb{C})$. The scalar 2-coboundary is defined according to \eqref{def cob}by
\begin{eqnarray}
\delta^2(f)(x_{0},x_{1},x_{2})=-f([x_{0},x_{1}],\alpha(x_{2}))+(-1)^{|x_{2}||x_{1}|}f([x_{0},x_{2}],\alpha(x_{1}))+f(\alpha(x_{0}),[x_{1},x_{2}]) .\ \ \ \ \, \label{dddd}
\end{eqnarray}
Now, we suppose that $f$ is a $ 2$-cocycle  of $osp(1,2)_{\lambda}$.  Then $f$ satisfies
\begin{eqnarray}
-f([x_{0},x_{1}],\alpha(x_{2}))+(-1)^{|x_{2}||x_{1}|}f([x_{0},x_{2}],\alpha(x_{1}))+f(\alpha(x_{0}),[x_{1},x_{2}])=0 .\label{ddd11}
\end{eqnarray}
By plugging the following 
 triples  
 \begin{eqnarray*}
  &(H,X,F),\ (H,X,Y),\ (H,X,G),\ (H,Y,G),\ (X,Y,F),\ (X,F,G)), (Y,F,G), \\ & \ (H,Y,F),\ (X,Y,G),\ (H,F,G),\ (H,F,F),\ (H,G,G),\ (X,G,G)
  \end{eqnarray*} respectively in  \eqref{ddd11}
 we obtain
\begin{eqnarray*}
 & f(H,G)=f(X,F), \ f(G,X)=0,\ f(H,F)=f(G,Y) , \ f(G,G)=f(X,H),\\ & f(F,F)=f(Y,H)\  f(F,Y)=0,\ f(X,Y)=f(G,F).
\end{eqnarray*}
So, if we consider the map $g:osp(1,2)\rightarrow \mathbb{R}$ defined  by 
\begin{eqnarray*}
&g(X)=\frac{1}{2\lambda^{2}} f(H,X),   \ g(Y  )=-\frac{\lambda^{2}}{2}f(H,X),\   g(F)=-\lambda f(H,F), \\ & \   g(G)
=\frac{1}{\lambda} f(H,G), \ g(H)=f(X,Y),
\end{eqnarray*}
 we obtain
\begin{eqnarray*}
&& f(a_1H+a_2X+a_3Y+a_4F+a_5G,b_1H+b_2X+b_3Y+b_4F+b_5G)\\
&& \quad =\delta(g)(a_1H+a_2X+a_3Y+a_4F+a_5G,b_1H+b_2X+b_3Y+b_4F+b_5G)
\end{eqnarray*}

Therefore $H^2(osp(1,2)_{\lambda},\mathbb{C})=\{0\}.$

Notice that this result is the same for any $r\geqslant 1$.

\end{example}
\section{Extensions of Hom-Lie superalgebras}
The   extension theory of Hom-Lie algebras algebras was stated first in\cite{HLS,LS1}.\\
An extension of a Hom-Lie superalgebra $(\mathcal{G},[.,.],\alpha)$ by Hom-module  $(V,\alpha_{V})$ is an exact sequence

$$0\longrightarrow (V,\alpha_{V})\stackrel{i}{\longrightarrow} (\widetilde{\mathcal{G}},\widetilde{\alpha})\stackrel{\pi}{\longrightarrow }(\mathcal{G},\alpha) \longrightarrow 0 $$ satisfying $\widetilde{\alpha}\ o\  i=i\ o\ \alpha_{V} $ and $\alpha \ o \ \pi = \pi\  o \ \widetilde{\alpha}.$\\
We say that the extension is central if $[\widetilde{\mathcal{G}},i(V)]_{\widetilde{\mathcal{G}}}=0.$\\
Two extensions $$0 \longrightarrow  (V,\alpha_{V})\stackrel{i_{k}}{\longrightarrow}      (\mathcal{G}_{k},\alpha_{k})\stackrel{\pi_{k}}{\longrightarrow } (\mathcal{G},\alpha) \longrightarrow 0 \ \ \ (k=1,2)$$ are equivalent if there is an isomorpism $\varphi:(\mathcal{G}_{1},\alpha_{1})\rightarrow (\mathcal{G}_{2},\alpha_{2})$ such that $\varphi o \ i_{1}= i_{2}$ and $\pi_{2}\ o\ \varphi=\pi_{1}.$

\subsection{Trivial representation of Hom-Lie superalgebras}
Let $V=\mathbb{C}\ (\textrm {or } \mathbb{R})$  and $[.,.]_{V}=0.$ Obviously, $\forall  \beta \in End(\mathbb{C}),\ (\mathcal{G},[.,.]_{V},\beta)$ is a representation of the Hom-Lie superalgebra $(\mathcal{G},[.,.],\alpha)$. This  representation is called   trivial representation of the Hom-Lie superalgebra $(\mathcal{G},[.,.],\alpha).$\\
In the following we consider central extensions of a Hom-Lie superalgebra $(\mathcal{G},[.,.],\alpha)$. We will see that it is controlled by the second cohomologie $H^{2}(\mathcal{G},V).$
Let $\theta \in C^{2}_{\alpha}(\mathcal{G},V),$ we consider the direct sum $\widetilde{\mathcal{G}}=\widetilde{\mathcal{G}_{0}}\oplus \widetilde{\mathcal{G}_1} $ where $\widetilde{\mathcal{G}_{0}}=\mathcal{G}_{0}\oplus \mathbb{C}$ and $\widetilde{\mathcal{G}_{1}} =\mathcal{G}_{1}$ with the following bracket $$ \Big[(x,s), (y,t)\Big]_{\theta}=\Big([x,y],\theta(x,y)\Big)\ \ \ \forall x,y \in \mathcal{G} \     s, t\in \mathbb{C}.$$
Define $ \widetilde{\alpha}:\widetilde{\mathcal{G}} \rightarrow \widetilde{\mathcal{G}}$ by $\widetilde{\alpha}(x,s)=(\alpha(x),s),$
\begin{thm}
The triple $(\widetilde{\mathcal{G}}, [.,.]_{\theta},\widetilde{\alpha})$   is a Hom-Lie superalgebra if and only if $\theta$ is a $2$-cocycle (i.e. $\delta^{2}(\theta)=0$).\\We call the Hom-Lie superalgebra $(\widetilde{\mathcal{G}}, [.,.]_{\theta},\widetilde{\alpha})$ the central extension of $(\mathcal{G},[.,.],\alpha)$ by $\mathbb{C}$.
\end{thm}
\begin{proof}
$\widetilde{\alpha}$ is a morphism with respect to the bracket $[.,.]_{\theta}$ follows from the fact that $\theta \circ \alpha =\theta.$ More precisely, we have $$\widetilde{\alpha}[(x,s),(y,t)]_{\theta}=(\alpha[x,y],\theta(x,y)).$$
On the other hand, we have $$ [\widetilde{\alpha}(x,s),\widetilde{\alpha}(y,t)]_{\theta}=[(\alpha(x),s),(\alpha(y),t)]_{\theta}=([\alpha(x),\alpha(y)],\theta(\alpha(x),\alpha(y))).$$
Since $\alpha$ is a morphism and $\theta(\alpha(x),\alpha(y))=\theta(x,y),$ we deduce that $\widetilde{\alpha}$ is a morphism.\\
By direct computations, we have
\begin{eqnarray*}
   &&\circlearrowleft_{(x,s),(y,t),(z,m)}(-1)^{|(x,s)||(z,m)|}\Big[\widetilde{\alpha}(x,s),[(y,t),(z,m)]_{\theta}\Big]_{\theta}\\&=&
   \circlearrowleft_{(x,s),(y,t),(z,m)}(-1)^{|x||z|}\Big[(\alpha(x),s),([y,z],\theta(y,z))\Big]_{\theta}\\
   &=& \circlearrowleft_{x,y,z}(-1)^{|x||z|}\Big([\alpha(x),[y,z]],\theta(\alpha(x),[y,z])\Big)\\
   &=& \circlearrowleft_{x,y,z}(-1)^{|x||z|} \Big(0,\theta(\alpha(x),\theta(\alpha(x),[y,z])\Big).
\end{eqnarray*}
Thus, by Hom-Jacobi identity of $\mathcal{G},$ the bracket  $ [.,.]_{\theta}$ satisfies the Hom-Jacobi identity if and only if $$\circlearrowleft_{x,y,z}(-1)^{|x||z|}\theta(\alpha(x),[y,z])=0.$$ That means that $\delta^{2}\theta=0.$

\end{proof}

Finally, remember that $i$ (resp. $\pi$) is an even morphism of Hom-Lie  superalgebra injective (resp. surjective).
\subsection{Cohomology space $H^{2}(\mathcal{G},V)$ and  Central extensions}\label{cohomologie2}
\begin{prop}
Let $(\mathcal{G},[.,.],\alpha)$ be a multiplicative  Hom-Lie superalgebra and $V$ be a $\mathcal{G}$-Hom-module. The second cohomology space $H^{2}(\mathcal{G},V)=Z^{2}(\mathcal{G},V)/ B^{2}(\mathcal{G},V)$ is  in one-to-one correspondence with the set of the equivalence classes of  central extensions of $(\mathcal{G},\alpha)$ by $(V,\beta).$
\end{prop}
\begin{proof}
Let $$0\longrightarrow (V,\beta)\stackrel{i}{\longrightarrow} (\widetilde{\mathcal{G}},\widetilde{\alpha})\stackrel{\pi}{\longrightarrow }(\mathcal{G},\alpha) \longrightarrow 0,$$ be  a central extension of Hom-Lie superalgebra $(\mathcal{G},\alpha)$ by $(V,\alpha_{V}),$
so there is a  space $H$ such that $\widetilde{\mathcal{G}}=H \oplus i(V).$

The map  $\pi_{/H}:H\rightarrow \mathcal{G}$ (resp $k:V\rightarrow i(V) $) defined by $\pi_{/H}(x)=\pi(x) $ (resp. $k(v)=i(v)$)is bijective, its inverse s (resp. $l$) note.  Considering the map $\varphi:\mathcal{G}\times V \rightarrow \widetilde{\mathcal{G}}$  defined  by
  $ \varphi(x,v)= s(x)+i (  v),$    it is easy to verify that $\varphi$ is a bijective.\\
  Since $\pi$ is homomorphism of Hom-Lie superalgebras then $\pi \Big([s(x),s(y)]_{\widetilde{\mathcal{G}}}-s([x,y])\Big)=0$\\ so $[s(x),s(y)]_{\widetilde{\mathcal{G}}}-s([x,y]\in i(V).$\\
We set $[s(x),s(y)]-s([x,y])=G(x,y)\in i(V)$  then $F(x,y)=l\circ G(x,y) \in V,$  it easy to see  that $F(x,x)=0$ then $F \in C^2(\mathcal{G},V)$ is a $2$-cochain that defines a bracket on $\widetilde{\mathcal{G}}.$ In fact, we can identify as a superspace $L\times V$ and $\widetilde{\mathcal{G}}$ by $\varphi :(x,v)\rightarrow s(x)+i(v)$ where the bracket is $$[s(x)+i(v),s(y)+i(w)]_{\widetilde{\mathcal{G}}}=[s(x),s(y)]_{\widetilde{\mathcal{G}}}=s([x,y])+F(x,y).$$  Viewed  as elements of $\mathcal{G}\times V$ we have   $\Big[(x,v),(y,w)\Big]=\Big([x,y],F(x,y)\Big)$ and the homogeneous elements $(x,v)$ of $\mathcal{G}\times V$ are such that $|x|=|v|$ and we have in this case $|(x,v)|=|x|$.\\
We deduce that for every central extension
$$0\longrightarrow (V,\beta)\stackrel{i}{\longrightarrow} (\widetilde{\mathcal{G}},\widetilde{\alpha})\stackrel{\pi}{\longrightarrow }(\mathcal{G},\alpha) \longrightarrow 0 .$$  One may associate a two cocycle $F\in Z^{2}(\mathcal{G},V)$. Indeed, for $x,y \in \mathcal{G} ,$ if we set $$F(x,y)=l\Big([s(x),s(y)]-s([x,y]\Big)\in V,$$
then, we have $F(x,y)\in V$ and $F$ satisfies the $2$-cocycle conditions.\\

Conversely, for each $f\in Z^{2}(\mathcal{G},V),$  one can define a central extension
$$0\longrightarrow (V,\beta)\longrightarrow (\mathcal{G}_{f},\alpha_{f})\longrightarrow (\mathcal{G},\alpha) \longrightarrow 0 ,$$
by $$[(x,v),(y,w)]_{f}=([x,y],f(x,y)),$$ where $x, \ y \in \mathcal{G}$ and $v,\ w \in V.$\\
Let $f$ and $g$ be two elements of $ Z^{2}(\mathcal{G},V)$ such that $f-g \in B^{2}(\mathcal{G},V)$ i.e. $(f-g)(x,y)=h([x,y]),$ where $h:\mathcal{G}\rightarrow V $ is  a linear map satisfying $h\circ \alpha=\beta\circ h$. Now we prove that the extensions  defined by $f$ and $g$ are equivalent. Let us define $\Phi:\mathcal{G}_{f}\rightarrow \mathcal{G}_{g}$ by $$ \Phi (x,y)=(x,v-h(x)). $$
It is clear that $\Phi$ is bijective. Let us check that $\Phi$ is a homomorphism of Hom-Lie superalgebras. We have
\begin{eqnarray*}
  [\Phi((x,v)),\Phi(((y,w))]_{g}&=&[(x,v-h(x)),(y,w-h(y))]_{g}\\
  &=&([x,y],g(x,y))\\
  &=&([x,y],f(x,y)-h([x,y]))\\
  &=&\Phi(([x,y],f(x,y)))\\
  &=&\Phi([(x,v),(y,w)]_{f})
\end{eqnarray*}
and
\begin{eqnarray*}
\Phi\circ \widetilde{\alpha}((x,v))&=& \Phi(\alpha(x),\beta(v))\\
&=&(\alpha(x),\beta(v)-h(\alpha(x)))\\
&=&(\alpha(x),\beta(v)-\beta \circ h(x))\\
&=&(\alpha(x),\beta(v- h(x)))\\
&=&\widetilde{\alpha}\circ \Phi (x,v).
\end{eqnarray*}
Next, we show that for $f,g \in Z^{2}(\mathcal{G},V)$ such that the central extensions\\ $0\rightarrow (V,\beta)\rightarrow (\mathcal{G}_{f},\widetilde{\alpha} )\rightarrow (\mathcal{G},\alpha) \rightarrow 0 ,$ and $0\rightarrow (V,\beta)\rightarrow (\mathcal{G}_{g},\widetilde{\alpha} )\rightarrow (\mathcal{G},\alpha) \rightarrow 0 $ are equivalent, we have $f-g\in B^{2}(\mathcal{G},V).$ Let $\Phi$ be a homomorphism of Hom-Lie superalgebras such that
\begin{displaymath}
\xymatrix { 0 \ar[r] & (V,\beta) \ar[d]_{id_{V}}\ar[r]^{i_{1}} & (\mathcal{G}_{f},\widetilde{\alpha} ) \ar[d]_{\Phi} \ar[r]^{\pi_{1}}&(\mathcal{G},\alpha)\ar[r] \ar[d]_{id_{\mathcal{G}}}&0 \\
0\ar[r] & (V,\beta) \ar[r]^{i_{2}}     & (\mathcal{G}_{g},\widetilde{\alpha}         ) \ar[r]^{\pi_{2}} &(\mathcal{G},\alpha)\ar[r]&0}
\end{displaymath}
commutes. We can express $\Phi(x,v)=(x,v-h(x))$ for some linear map $h:\mathcal{G}\rightarrow V.$ Then we have
\begin{eqnarray*}
\Phi([(x,v),(y,w)]_{f})&=&\Phi(([x,y],f(x,y)))\\
&=&([x,y],f(x,y)-h([x,y])),
\end{eqnarray*}
\begin{eqnarray*}
[\Phi((x,v)),\Phi((y,w))]_{g}&=&[(x,v-h(x)),(y,w-h(y))]_{g}\\
&=&([x,y],g(x,y)),
\end{eqnarray*}
and thus $(f-g)(x,y)=h([x,y])$ (i.e. $f-g \in B^{2}(\mathcal{G},V)$), so we have completed the proof.
\end{proof}
\subsection{The adjoint representation of Hom-Lie superalgebras}
In this section we generalize some results of \cite{sheng}.\\
Let $(\mathcal{G},[.,.],\alpha)$ be a multiplicative  Hom-Lie superalgebra. We consider $\mathcal{G}$ as  a representation on itself via the  bracket and with respect to the morphism $\alpha$ .
\begin{defn}
The $\alpha^{s}$-adjoint representation of the  Hom-Lie superalgebra $(\mathcal{G},[.,.],\alpha),$  which we denote by $ad_{s},$ is defined by
 $$ad_{s}(a)(x)=[\alpha^{s}(a),x],\ \forall \ a,\ x \in \mathcal{G}.$$
\end{defn}
\begin{lem}
With the above notations, we have a $(\mathcal{G},ad_{s}(.)(.),\alpha)$ is a representation of the Hom-Lie superalgebra $\mathcal{G}.$

\end{lem}
\begin{proof}
The result follows from
\begin{eqnarray*}
 ad_{s}(\alpha(a))( \alpha(x)) &=&  [\alpha^{s+1}(a),\alpha(x)]\\
 &=&\alpha([\alpha^{s}(a),x])\\
 &=&\alpha \circ ad_{s}(a)(x),
\end{eqnarray*}
and
\begin{eqnarray*}
ad_{s}([x,y] )(\alpha (z))&=& [\alpha^{s}([x,y]),\alpha(z)]\\
&=& [[\alpha^{s}(x),\alpha^{s}(y)],\alpha(z)]\\
&=& -(-1)^{|z||[x,y]|}[\alpha(z),[\alpha^{s}(x),\alpha^{s}(y)]]\\
&=&(-1)^{|z||x|}(-1)^{|z||x|}[\alpha^{s+1}(x),[\alpha^{s}(y),z]]+(-1)^{|z||x|}(-1)^{|y||x|}[\alpha^{s+1}(y),[z,\alpha^{s}(x)]]\\
&=&[\alpha^{s+1}(x),[\alpha^{s}(y),z]]-(-1)^{|x||y|}[\alpha^{s+1}(y),[\alpha^{s}(x),z]].
\end{eqnarray*}
\end{proof}
The set of $k$-hom-cochains on $\mathcal{G}$ with coefficients in $\mathcal{G},$ which we denote by $C_{\alpha}^{k}(\mathcal{G};\mathcal{G}),$  is given by
$$C_{\alpha}^{k}(\mathcal{G};\mathcal{G})=\{f\in C^{k}(\mathcal{G};\mathcal{G}):\  f\circ \alpha= \alpha\circ f\}.$$ In particular, the set of $0$-Hom-cochains is given by:$$C_{\alpha}^{0}(\mathcal{G};\mathcal{G})=\{x\in \mathcal{G}:\  \alpha(x)=x\}.$$
\begin{prop}\label{adj}
With respect  to the  $\alpha^{s}$-adjoint representation $ad_{s},$ of the Hom-Lie superalgebra $(\mathcal{G},[.,.],\alpha),$ $D\in C_{\alpha,ad_{s}}^{1}$
is a $1$-cocycle if and only if $D$ is an $\alpha^{s+1}$-derivation of  the Hom-Lie superalgebra $(\mathcal{G},[.,.],\alpha),$ i.e. $D\in Der_{\alpha^{s+1}}(\mathcal{G}).$
\end{prop}
\begin{proof}
The conclusion follows directly from the definition of the coboundary operator $\delta.$ $D$ is closed if and only if
$$\delta(D)(x,y)=-D([x,y])+(-1)^{|x||D|}[\alpha^{s+1}(x),D(y)]+(-1)^{1+|y|(|D|+|x|)}[\alpha^{s+1}(y),D(x)]=0,$$
so $$D([x,y])=[D(x),\alpha^{s+1}(y)]+(-1)^{|x||D|}[\alpha^{s+1}(x),D(y)],$$ wich implies that $D$ is an $\alpha^{s+1}$-derivation.
\end{proof}
\subsubsection{The $\alpha^{-1}$-adjoint representation $ad_{-1}$}
\begin{prop}
With respect to the $\alpha^{-1}$-adjoint representation $ad_{-1}$, we have
\begin{eqnarray*}
H^{0}(\mathcal{G},\mathcal{G})&=&C^{0}_{\alpha}(\mathcal{G};\mathcal{G})=\{x\in \mathcal{G}:\  \alpha(x)=x\};\\
H^{1}(\mathcal{G},\mathcal{G})&=&Der_{\alpha^{0}}(\mathcal{G}).
\end{eqnarray*}
\end{prop}
\begin{proof}
For any $0$-hom-cochain $x \in C^{0}_{\alpha}(\mathcal{G};\mathcal{G}),$ we have $\delta(x)(y)=(-1)^{|y||x|}[\alpha^{-1}(y),x]=0,\ \forall y\in \mathcal{G} .$\\
Therefore, any $0$-hom-cochain is closed. Thus, we have $H^{0}(\mathcal{G},\mathcal{G})=C^{0}_{\alpha}(\mathcal{G};\mathcal{G})=\{x\in \mathcal{G}: \ \alpha(x)=x\}.$
Since there is no exact $1$-hom-cochain, by Proposition \ref{adj}, we have $H^{1}(\mathcal{G},\mathcal{G})=Der_{\alpha^{0}}(\mathcal{G}).$
\end{proof}
Let $\omega \in C^{2}_{\alpha}(\mathcal{G};\mathcal{G})$ be an even super-skew-symmetric bilinear operator commuting with $\alpha.$ Considering a $t$-parametrized family of bilinear operations $$[x,y]_{t}=[x,y]+t\omega (x,y).$$
Since $\omega$ commutes with $\alpha,$ $\alpha$ is a morphism with respect to the brackets $[.,.]_{t}$ for every $t.$ If  $(\mathcal{G}[[t]],[.,.]_{t},\alpha)$ is a Hom-Lie superalgebra, we say that $\omega$ generates a deformation of the Hom-Lie superalgebra $(\mathcal{G},[.,.],\alpha).$ The  super Hom-Jacobi identity of $[.,.]_{t}, $,  is  equivalent to
\begin{eqnarray}
\circlearrowleft_{x,y,z}(-1)^{|x||z|}\Bigg(\omega\Big(\alpha(x),[y,z]\Big)+\Big[\alpha(x),[y,z]\Big]\Bigg)&=&0, \label{rep-1}\\
\circlearrowleft_{x,y,z}(-1)^{|x||z|}\omega\Big(\alpha(x),\omega(y,z)\Big)&=&0\label{rep-2}.
\end{eqnarray}
Obviously, (\ref{rep-1}) means that $\omega$ is even $2$-cycle with respect to the $\alpha^{-1}$-adjoint representation $ad_{-1}.$ Furthermore, (\ref{rep-2}) means that $\omega$ must itself defines a Hom-Lie superalgebra structure on $\mathcal{G}.$
\subsubsection{The $\alpha^{0}$-adjoint representation $ad_{0}$}

\begin{prop}
With respect  to the $\alpha^{0}$-adjoint representation $ad_{0},$ we have
$$H^{0}(\mathcal{G};\mathcal{G})=\{x\in \mathcal{G}:\  \alpha(x)=x, [x,y]=0, \forall y\in \mathcal{G}\},$$
$$H^{1}(\mathcal{G};\mathcal{G})=Der_{\alpha}(\mathcal{G}):\ Inn_{\alpha}(\mathcal{G}).$$
\end{prop}
\begin{proof}
For any $0$-hom-cochain
 we have $d_{0}x(y)=[\alpha^{0}(y),x]=[x,y]. $\\
Therefore, the set of $0$-cycles  $Z^{0}(\mathcal{G},\mathcal{G})$ is given by $Z^{0}(\mathcal{G},\mathcal{G})=\{x\in C_{\alpha}^{0}(\mathcal{G},\mathcal{G}):\ [x,y]=0, \ \forall y\in\mathcal{G}\}.$\\ Since $B^{0}(\mathcal{G},\mathcal{G})=\{0\}$, we deduce that
 $H^{0}(\mathcal{G};\mathcal{G})=\{x\in \mathcal{G}:\ \alpha(x)=x, [x,y]=0, \forall y\in \mathcal{G}\}.$\\
 For any $f\in C_{\alpha}^{1}(\mathcal{G},\mathcal{G}),$ we have $$\delta(f)(x,y)=-f([x,y])+(-1)^{|x||f|}[\alpha(x),f(y)]+(-1)^{1+|y|(|f|+|x|)}[\alpha(y),f(x)]$$ so,
 $$\delta(f)(x,y)=-f([x,y])+[f(x),\alpha(y)]+(-1)^{|x||f|}[\alpha(x),f(y)].$$ Therefore, The set of
$1$-cocycle $Z^{0}(\mathcal{G},\mathcal{G})$ is exactly the set of $\alpha$-derivation $Der_{\alpha}.$\\
Furthermore, it is obvious that any exact $1$-coboundary is of the form  of $[x,.]$ for some $x\in C_{\alpha}^{0}(\mathcal{G};\mathcal{G}).$
Therfore, we have $B^{1}(\mathcal{G},\mathcal{G})=Inn_{\alpha}(\mathcal{G}).$ Wich implies that $H^{1}(\mathcal{G};\mathcal{G})=Der_{\alpha}(\mathcal{G})/Inn_{\alpha}(\mathcal{G}).$
\end{proof}

\subsubsection{The coadjoint representation $ad^{*}$}
Let $(\mathcal{G},[.,.],\alpha)$ be a Hom-Lie superalgebra and $(\mathcal{G},[.,.]_{V},\beta)$ b a representation of  $\mathcal{G}.$ Let $V^{*}$ be the dual vector space of $V.$ We define an even bilinear map  $[.,.]_{V^{*}}:\mathcal{G}\times V^{*}\rightarrow V^{*}$
by $$[x,f]_{V^{*}}(v)=-f([x,v]_{V}),\ \forall x\in \mathcal{G}\ f\in V^*, \textrm{ and } v\in V.$$
Let $f\in V^*,\ x,y\ \in \mathcal{G}$ and $v\in V.$  We compute the right hand side of the identity (\ref{rep-2})
\begin{eqnarray*}
   [\alpha(x),[y,f]_{V^*}]_{V^*}(v)-(-1)^{|x||y|}[\alpha(y),[x,v]_{V^*}]_{V^*}&=&- [y,f]_{V^*}([\alpha(x),v]_*)+(-1)^{|x||y|} [x,f]_{V^*}([\alpha(y),v]_{V^*})\\
   &=&f([y,[\alpha(x),v]_{V}]_{V})-(-1)^{|x||y|}f([x,[\alpha(y),v]_{V}]_{V}).
\end{eqnarray*}
On the other hand, we set that the twisted map for $[.,.]_{V^{*}}$ is $\beta^*=^t\beta,$ then the left hand side of the identity (\ref{rep-2}) writes
\begin{eqnarray*}
[[x,y],\beta^*(f)]_{V^*}(v)=-\beta^*(f)([[x,y],v]_{V})=-^t \beta (f)([[x,y],v]_{V})&=&-f \circ \beta ([[x,y],v]_{V}).
\end{eqnarray*}
Therefore, we have the following proposition:
\begin{prop}
Let $(\mathcal{G},[.,.],\alpha)$ be a Hom-Lie superalgebra and $(\mathcal{G},[.,.]_{V},\beta)$ be a representation  of $\mathcal{G}.$ The triple $(V^*,[.,.]_{V^{*}},\beta^*),$ where $[x,f]_{V^{*}}(v)=-f([x,v]_{V}),\ \forall x\in \mathcal{G},\ f\in V^*, \ v\in V$ defines a representation of the Hom-Lie superalgebra $\mathcal{G},[.,.],\alpha$ if only and only if $$[[x,y],\beta(v)]_{V}=(-1)^{|x||y|}[x,[\alpha(y),v]_{V}]_{V}-[y,[\alpha(x),v]_{V}]_{V}.$$
\end{prop}
We obtain, the following characterization in the case of adjoint representation.
\begin{cor}
Let $(\mathcal{G},[.,.],\alpha)$ be a Hom-Lie superalgebra and  $(\mathcal{G}, ad ,\alpha)$ be the adjoint representation of $\mathcal{G},$ where $ad:\mathcal{G} \rightarrow End(\mathcal{G}).$ We set $ad^*:\mathcal{G} \rightarrow End(\mathcal{G}^*)$ and $ad^*(x)(f)=-f\circ ad(x).$\\
Then $(\mathcal{G}^*,ad^*,\alpha^*)$ is a representation of $\mathcal{G}$ if and only if$$[[x,y],\alpha(z)]=(-1)^{|x||y|}[x,[\alpha(y),z]]-[y,[\alpha(x),z]], \ \ \ \forall x,\ y,\ z\in \mathcal{G}.$$
\end{cor}

\section{ Cohomology of the  $q$-Witt superalgebra}
In the following, we describe $q$-Witt Hom-Lie superalgebra obtained in \cite{AmmarMakhloufJA2010} and we compute its derivations and second cohomology group. 

Let $\mathcal{A}=\mathcal{A}_{0}\oplus\mathcal{A}_{1}$ be an associative superalgebra. We assume that $\mathcal{A}$ is super-commutative, that is for homogeneous elements $a, b$ the identity $ab=(-1)^{|a||b|}ba$ holds. For example, $\mathcal{A}_{0}=\mathbb{C}[t,t^{-1}]$ and $\mathcal{A}_{1}=\theta\mathcal{A}_{0}$ where $\theta$ is the Grassman variable $(\theta^{2}=0).$
Let  $q\in \mathbb{C}\backslash \{0,1\}$ and $n\in \mathbb{N},$ we set $\{n\}=\frac{1-q^{n}}{1-q}, $ a $q$-number.
Let $\sigma$ be the algebra endomorphism on $\mathcal{A}$ defined by  $$\sigma(t^{n})=q^{n}t^{n} \ \ and \ \ \sigma (\theta)=q\theta.$$
Let $\partial_{t}$ \ and \ $\partial_{\theta}$\ be two linear maps on $\mathcal{A}$ defined by $$\partial_{t}(t^{n})=\{n\}t^{n}, \ \partial_{t}(\theta t^{n})=\{n\}\theta t^{n},$$ $$\partial_{\theta}(t^{n})=0, \ \ \partial_{\theta}(\theta t^{n})=q^{n}t^{n}.$$

\begin{defn}
Let $i\in \mathbb{Z}_{2}$. A $\sigma$-derivation $D_{i}$ on $ \mathcal{A}$ is an endomorphism satisfying:
$$D_{i}(ab)=D_{i}(a)b+(-1)^{i|a|}\sigma (a)D_{i}(b) $$
where $ a, b\ \in \mathcal{A}$ are homogeneous element and $|a|$ is the parity of a.\\
A $\sigma$-derivation $D_0$ is called even $\sigma$-derivation  and $D_1$ is called odd $\sigma$-derivation. The set of
all $\sigma$-derivations is denoted by $Der_{\sigma}( \mathcal{A})$. Therefore, $Der_{\sigma}(\mathcal{A})=Der_{\sigma}(\mathcal{A})_{0}\oplus Der_{\sigma}(\mathcal{A})_{1}$, where
$Der_{\sigma}(\mathcal{A})_{0}$ (resp $Der_{\sigma}(\mathcal{A})_{1}$) is the space of even (resp. odd) $\sigma$-derivations.

\end{defn}

\begin{lem}
The linear map $\Delta=\partial_{t}+\theta \partial_{\theta}$ on  $ \mathcal{A}$ is an even $\sigma$-derivation.\\Hence,$$\Delta(t^{n})=\{n\}t^{n},$$
$$\Delta(\theta t^{n})=\{n+1\}\theta t^{n}.$$
\end{lem}

Let $\mathcal{W}^q= \mathcal{A}.\Delta,$ be a superspace generated by the elements $L_{n}=t^{n}.\Delta$ of parity $0$ and the elements $G_{n}=\theta t^{n}.\Delta$ of parity 1.\\
Let $[-,-]_{\sigma}$ be a bracket on the superspace $\mathcal{W}^{q}$ defined by
\begin{eqnarray}
&&[L_{n},L_{m}]_{\sigma}=(\{m\}-\{n\})L_{n+m}, \label{crochet1} \\
&&[L_{n},G_{m}]_{\sigma}=(\{m+1\}-\{n\})G_{n+m}.\label{crochet2}
 \end{eqnarray}
 The others brackets are obtained by supersymmetry or are $0.$\\
It is easy to see that $\mathcal{W}^{q}$ is a $\mathbb{Z}-$graded algebra:
\begin{eqnarray*}
 \mathcal{W}^{q}&=&\oplus_{n\in\mathbb{Z}}\mathcal{W}^{q}_n,
 \end{eqnarray*}
 where
 \begin{eqnarray*}
 \mathcal{W}^{q}_{n}=span_{\mathbb{C}} \{L_{n}, G_{n}\}.
  \end{eqnarray*}
Let $\alpha$ be an even linear map on $\mathcal{W}^{q}$ defined on the generators by $$\alpha(L_{n})=(1+q^{n})L_{n},$$ $$\alpha(G_{n})=(1+q^{n+1})G_{n}.$$
\begin{prop}[\cite{AmmarMakhloufJA2010}]
The triple $(\mathcal{W}^{q},[-,-]_{\sigma},\alpha)$ is a Hom-Lie superalgebra.
\end{prop}


\subsection{Derivations of  the Hom-Lie superalgebra $\mathcal{W}^{q}$}

 A homogeneous $\alpha^{k}$-derivation is said  of degree $s$ if there exists  $s\in \mathbb{Z}$ such that for all $n\in  \mathbb{Z}$ we have $D(<L_n>)\subset <L_{n+s}>$. 
 The corresponding subspace of  homogeneous $\alpha^{k}$-derivations of degree $s$ is denoted by $Der_{\alpha^{k},i}^{s} \ (i\in \mathbb{Z}_{2})$.\\
 It easy to check that $\displaystyle Der_{\alpha^{k}}(\mathcal{W}^{q})=\oplus_{s\in \mathbb{Z}}\big(Der_{\alpha,0}^{s}(\mathcal{W}^{q})\oplus Der_{\alpha,1}^{s}(\mathcal{W}^{q})\big).$\\
 Let $D$ be  a homogeneous $\alpha^{k}$-derivation 
\begin{eqnarray*}
D([x,y])=[D(x),  \alpha^{k}(y)]+(-1)^{|x||D|}[\alpha^{k}(x),D(y)]\ \ for \ all\  homogeneous \ x,\ y \in \mathcal{W}^{q}.
\end{eqnarray*}
We deduce that
\begin{eqnarray}
(\{m\}-\{n\})D(L_{n+m})=(1+q^{m})^{k}[D(L_{n}),L_{m}]+(1+q^{n})^{k}[L_{n},D(L_{m})] \label{der zero}
\end{eqnarray}
and
\begin{eqnarray}
(\{m+1\}-\{n\}) D(G_{n+m})=(1+q^{m+1})^{k}[D(L_{n}),G_{m}]+(1+q^{n})^{k}[L_{n},D(L_{m})] \quad   \forall n,\ m \in \mathbb{Z}. \nonumber \label{der 0}\\
\end{eqnarray}

\subsubsection{The $\alpha^{0}$-derivation of the  Hom-Lie superalgebra $\mathcal{W}^{q}$}
\begin{prop}
The set of $\alpha^{0}$-derivations of the Hom-Lie superalgebra $\mathcal{W}^{q}$ is $$Der_{\alpha^{0}}(\mathcal{W}^{q})=<D_{1}>\oplus<D_{2}>$$ where $D_{1}$ and $D_2$ are defined, with respect to the basis as 
\begin{eqnarray*}
D_{1}(L_{n})=nL_n, \ D_{1}(G_n)=G_n,\\ 
D_{2}(L_{n})=nG_{n-1},\ D_{2}(G_n)=L_{n-1}.
\end{eqnarray*}
\end{prop}
\begin{proof}

We consider The two cases $|D|=0$ and $|D|=1$.
\textbf{Case 1: $|D|=0$} \\
Let $D$  ba an  even derivation of degree $s$, $D(L_n)=a_{s,n}L_{s+n}$ and $D(G_{n})=b_{s,n}G_{s+n}$, by (\ref{der zero}) we have
\begin{eqnarray*}
(\{m\}-\{n\})a_{s,n+m}
=(\{m\}-\{n+s\})a_{s,n}+(\{m+s\}-\{n\})a_{s,m}.
\end{eqnarray*}
We deduce that
\begin{eqnarray*}
(q^{n}-q^{m})a_{s,n+m}
=(q^{n+s}-q^{m})a_{s,n}+(q^{n}-q^{m+s})a_{s,m}.
\end{eqnarray*}
If $m=0$, we have,
\begin{eqnarray*}
q^{n}(1-q^{s})a_{s,n}
=(q^{n}-q^{s})
a_{s,0}.
\end{eqnarray*}
If $ s\neq 0$ we have
\begin{eqnarray*}
a_{s,n}
=\frac{1-q^{s-n}}{1-q^{s}}
a_{s,0}.
\end{eqnarray*}
We deduce that
\begin{eqnarray}
(q^{n}-q^{m})\frac{1-q^{s-n}}{1-q^{s}}
a_{s,0}
=(q^{n+s}-q^{m})\frac{1-q^{s-n}}{1-q^{s}}
a_{s,0}+(q^{n}-q^{m+s})\frac{1-q^{s-m}}{1-q^{s}}
a_{s,0}.\label{der 00}
\end{eqnarray}
Taking $n=2s$,\ $m=s$, in (\ref{der 00}) we obtain $a_{s,0}=0$, so $a_{s,n}=0$.\\
If $s=0$ and $n\neq m$ we have $a_{s,n}=na_{s,1}$.\\
By (\ref{der 0}) and
$D(G_{n})=b_{s,n}G_{n+s}$ we have
\begin{eqnarray*}
(\{m+1\}-\{n\})b_{s,n+m}=(\{m+s+1\}-\{n\})b_{s,m}
\end{eqnarray*}
So
\begin{eqnarray*}
(q^{n}-q^{m+1})b_{s,n+m}=(q^{n}-q^{m+s+1})b_{s,m}.
\end{eqnarray*}
Taking $n=0$, we have, $(q^{m+1}-q^{m+s+1})b_{s,m}=0$ then If $s\neq 0$ we have $b_{s,m}=0$.\\
If $s=0$ and $n\neq m+1$ we have $b_{s,n+m}=b_{s,m}$ so $b_{s,n}=b_{s,0}$.\\
Finally, it follows  that  the set of even $\alpha^{0}$-derivations is $\displaystyle Der_{\alpha^{0},0}(\mathcal{W}^{q})=Der_{\alpha,0}^{0}(\mathcal{W}^{q})=<D_1>$ with $D_1(L_{n})=nL_n$ and $D_1(G_n)=G_n$.\\

\textbf{Case 2: $|D|=1$} \\
Let $D$ be an  odd derivation  of degree $s$, $D(L_n)=a_{s,n}G_{s+n}$ and $D(G_{n})=b_{s,n}L_{s+n}$. By(\ref{der zero}) we have
\begin{eqnarray*}
(\{m\}-\{n\})a_{s,n+m}
=(\{m\}-\{n+s+1\})a_{s,n}+(\{m+s+1\}-\{n\})a_{s,m}.
\end{eqnarray*}
We deduce that,
\begin{eqnarray*}
(q^{n}-q^{m})a_{s,n+m}
=(q^{n+s+1}-q^{m})a_{s,n}+(q^{n}-q^{m+s+1})a_{s,m}.
\end{eqnarray*}
If $m=0$, we have,
\begin{eqnarray*}
q^{n}(1-q^{s+1})a_{s,n}
=(q^{n}-q^{s+1})
a_{s,0}.
\end{eqnarray*}
If $ s\neq -1$ we have
\begin{eqnarray*}
a_{s,n}
=\frac{1-q^{s+1-n}}{1-q^{s+1}}
a_{s,0}.
\end{eqnarray*}
Then
\begin{eqnarray}
(q^{n}-q^{m})\frac{1-q^{s+1-n}}{1-q^{s+1}}
a_{s,0}
=(q^{n+s+1}-q^{m})\frac{1-q^{s+1-n}}{1-q^{s+1}}
a_{s,0}+(q^{n}-q^{m+s+1})\frac{1-q^{s+1-m}}{1-q^{s+1}}
a_{s,0}.\label{der 011}
\end{eqnarray}
Taking $n=2s+2$,\ $m=s+1$, in (\ref{der 011}) we obtain $a_{s,0}=0$, so $a_{s,n}=0$.\\
If $s=-1$ and $n\neq m$, then $a_{s,n}=na_{s,1}$.\\
By (\ref{der 0}) and
$D(G_{n})=b_{s,n}L_{n+s}$ we have
\begin{eqnarray*}
(\{m+1\}-\{n\})b_{s,n+m}&=&(\{m+s+1\}-\{n\})b_{s,m}.
\end{eqnarray*}
Taking $n=0$, we have $(q^{m+1}-q^{m+s+1})b_{s,m}=0$ then if $s\neq -1$, it follows  $b_{s,m}=0$.\\
If $s=-1$ and $n\neq m+1$, we obtain $b_{s,n+m}=b_{s,m}$ so $b_{s,n}=b_{s,0}$.\\
Finally,  it follows that  the set of odd $\alpha^{0}$-derivations is $\displaystyle Der_{\alpha^{0},1}(\mathcal{W}^{q})=Der_{\alpha,0}^{-1}(\mathcal{W}^{q})=<D_2>$ with $D_{2}(L_{n})=nG_{n-1}$ and $D_{2}(G_n)=L_{n-1}$.\\

\end{proof}

\subsubsection{ The $\alpha^{1}$-derivations of the Hom-Lie superalgebra $\mathcal{W}^{q}$}
\begin{prop}
If $D$ is an $\alpha$-derivation then $D=0$.
\end{prop}

 \begin{proof}

\textbf{Case 1: $|D|=0$} \\
Let $D$ be an even derivation  of degree $s$, $D(L_n)=a_{s,n}L_{s+n}$ and $D(G_{n})=b_{s,n}G_{s+n}$. 

By (\ref{der zero}) we have
\begin{eqnarray*}
(\{m\}-\{n\})a_{s,n+m}
=(1+q^{m})(\{m\}-\{n+s\})a_{s,n}+(1+q^{n})(\{m+s\}-\{n\})a_{s,m}.
\end{eqnarray*}
We deduce that,
\begin{eqnarray*}
a_{s,n+m}
=\frac{(1+q^{m})(q^{n+s}-q^{m})}{q^{n}-q^{m}}a_{s,n}+\frac{(1+q^{n})(q^{n}-q^{m+s})}{q^{n}-q^{m}}a_{s,m}.
\end{eqnarray*}
If $m=0$, we have,
\begin{eqnarray*}
a_{s,n}
=\frac{(1+q^{n})(q^{n}-q^{s})}{1+q^{n}-2q^{n+s}}
a_{s,0}.
\end{eqnarray*}
So
\begin{eqnarray*}
a_{s,n+m}
=\frac{(1+q^{n+m})(q^{n+m}-q^{s})}{1+q^{n+m}-2q^{n+m+s}}
a_{s,0}.
\end{eqnarray*}
Then
\begin{eqnarray*}
&&\frac{(1+q^{n+m})(q^{n+m}-q^{s})}{1+q^{n+m}-2q^{n+m+s}}
a_{s,0}\\
&=&\frac{(1+q^{n})(q^{n}-q^{m+s})(1+q^{m})(q^{m}-q^{s})}{(q^{n}-q^{m})(1+q^{m}-2q^{m+s})}
a_{s,0}
-\frac{(1+q^{m})(q^{m}-q^{n+s})(1+q^{n})(q^{n}-q^{s})}{(q^{n}-q^{m})(1+q^{n}-2q^{n+s})}a_{s,0}.\\
\end{eqnarray*}
If $q\in[0,1[$, and letting  $n,\ m\rightarrow +\infty$  we obtain $a_{s,0}=0$.\\
If $q>1$ and for a fixed $m=s$, then when $n$ goes to infinity we obtain   $a_{s,0}=0$.\\
We deduce that $D(L_n)=0$.\\

By (\ref{der 0}) and
$D(G_{n})=b_{s,n}G_{n+s}$ we have
\begin{eqnarray*}
(\{m+1\}-\{n\})b_{s,n+m}&=&(1+q^{n})(\{m+s+1\}-\{n\})b_{s,m}.
\end{eqnarray*}
So
\begin{eqnarray*}
(q^{n}-q^{m+1})b_{s,n+m}&=&(1+q^{n})(q^{n}-q^{m+s+1})b_{s,m}.
\end{eqnarray*}
Taking $n=0$, we have, $(1+q^{m+1}-2q^{m+s+1})b_{s,m}=0$.\\
Then t $b_{s,m}=0$ so $f(G_n)=0$.
 Hence $D\equiv 0$.\\
 \textbf{Case 2: $|D|=1$} \\
Let $D$ be an   odd derivation  of degree $s$, $D(L_n)=a_{s,n}G_{s+n}$ and $D(G_{n})=b_{s,n}L_{s+n}$. By (\ref{der zero}) we have
\begin{eqnarray*}
(\{m\}-\{n\})a_{s,n+m}
=(1+q^{m})(\{m\}-\{n+s+1\})a_{s,n}+(1+q^{n})(\{m+s+1\}-\{n\})a_{s,m}.
\end{eqnarray*}
Then
\begin{eqnarray*}
a_{s,n+m}
=\frac{(1+q^{n})(q^{n}-q^{m+s+1})}{(q^{n}-q^{m})}a_{s,m}-\frac{(1+q^{m})(q^{m}-q^{n+s+1})}{(q^{n}-q^{m})}a_{s,n}.\\
\end{eqnarray*}
 If $m=0$, we have,
\begin{eqnarray*}
a_{s,n}
=\frac{(1+q^{n})(q^{n}-q^{s+1})}{1+q^{n}-2q^{n+s+1}}
a_{s,0}.
\end{eqnarray*}
So
\begin{eqnarray*}
a_{s,n+m}
=\frac{(1+q^{n+m})(q^{n+m}-q^{s+1})}{1+q^{n+m}-2q^{n+m+s+1}}
a_{s,0}.
\end{eqnarray*}
 Then
\begin{eqnarray*}
&&\frac{(1+q^{n+m})(q^{n+m}-q^{s+1})}{1+q^{n+m}-2q^{n+m+s+1}}
a_{s,0}\\
&=&\frac{(1+q^{n})(q^{n}-q^{m+s+1})(1+q^{m})(q^{m}-q^{s+1})}{(q^{n}-q^{m})(1+q^{m}-2q^{m+s+1})}
a_{s,0}
-\frac{(1+q^{m})(q^{m}-q^{n+s+1})(1+q^{n})(q^{n}-q^{s+1})}{(q^{n}-q^{m})(1+q^{n}-2q^{n+s+1})}a_{s,0}.\\
\end{eqnarray*}
 If $q\in[0,1[$, and making  $n,\ m\rightarrow +\infty$, we obtain $a_{s,0}=0$.\\
If $q>1$ and setting  $m=s$, then if $n$ goes to infinity we obtain   $a_{s,0}=0$.\\
We deduce that $D(L_n)=0$.\\
By (\ref{der 0}) and
$D(G_{n})=b_{s,n}L_{n+s}$,  we obtain 
\begin{eqnarray*}
(\{m+1\}-\{n\})b_{s,n+m}L_{m+n+s}&=&(1+q^{n})(\{m+s\}-\{n\})b_{s,m}L_{m+s+n}.
\end{eqnarray*}
So
\begin{eqnarray*}
(q^{n}-q^{m+1})b_{s,n+m}&=&(1+q^{n})(q^{n}-q^{m+s})b_{s,m}.
\end{eqnarray*}
Taking $n=0$ leads to $(1+q^{m+1}-2q^{m+s})b_{s,m}=0$.\\
It turns out that  $b_{s,m}=0$ so $D(G_n)=0$. 
Hence $D\equiv 0$.
 \end{proof}

\subsubsection{The $q$-derivations of the Hom-Lie superalgebra $\mathcal{W}^{q}$}
In this section, we study the $q$-derivations  of $\mathcal{W}^{q}$.
The derivation algebra of $\mathcal{W}^{q}$ is denoted by  $Der \mathcal{W}^{q}.$ Since $\mathcal{W}^{q}$ is
$\mathbb{Z}_{2}$-graded Hom-Lie superalgebra, we have $$Der \mathcal{W}^{q}= (Der \mathcal{W}^{q})_{0} \oplus (Der \mathcal{W}^{q})_{1},$$
where $(Der \mathcal{W}^{q})_{0}=\{D\in Der \mathcal{W}^{q}:\  D((\mathcal{W}^{q})_{i})\subset(\mathcal{W}^{q})_{i}, i\in
\mathbb{Z}_{2} \},$ denote the set of even derivations of $\mathcal{W}^{q},$ and $(Der \mathcal{W}^{q})_{1}=\{D\in Der \mathcal{W}^{q}:\  D((\mathcal{W}^{q})_{i})\subset(\mathcal{W}^{q})_{i+1}, i\in
\mathbb{Z}_{2} \},$ denote the set of odd derivations of $\mathcal{W}^{q}.$\\ 
The space  $\mathcal{W}^{q}$ maybe viewed also as a $\mathbb{Z}$-graded space.
 Define $$(Der \mathcal{W}^{q})_{s}=\{D\in Der \mathcal{W}^{q}:\ D(\mathcal{W}^{q}_{n})\subset \mathcal{W}_{n+s}^{q}\}.$$
Then we have $\displaystyle Der \mathcal{W}^{q}=\oplus_{s\in \mathbb{Z}}(Der \mathcal{W}^{q})_{s}.$ Obviously, the $\mathbb{Z}$-graded
and $\mathbb{Z}_{2 }$-graded structures are compatible.\\ Moreover let $Der_{q} \mathcal{W}^{q}_{0}=\oplus_{s\in \mathbb{Z}}(Der \mathcal{W}^{q})'_{s},\ Der_q \mathcal{W}^{q}_{1}=\oplus_{s\in \mathbb{Z}}(Der \mathcal{W}^{q})''_{s},$ where $(Der \mathcal{W}^{q})'_{s}\oplus(Der \mathcal{W}^{q})''_{s}=(Der \mathcal{W}^{q})_{s}.$

\begin{defn}
An element 
  $\varphi\in (Der \mathcal{W}^{q})_{0}\cap(Der \mathcal{W}^{q})_{s} $ (resp.  $ \varphi\in(Der \mathcal{W}^{q})_{1}\cap(Der \mathcal{W}^{q})_{s})$ is a $q$-derivation  if

\begin{eqnarray}
  \varphi ([x,y]) &=&\frac{1}{1+q^{s}}\big([ \varphi (x),\alpha(y)]+[\alpha (x),\varphi(y)]\big) \label{derwit1}
\end{eqnarray}
\begin{eqnarray}
(\textrm {resp.} \  \varphi ([x,y]) &=&\frac{1}{1+q^{s+1}}\big([ \varphi (x),\alpha(y)]+(-1)^{|x|}[\alpha (x),\varphi(y)]\big) \ \ ).\label{derwit2}
\end{eqnarray}
where $x, y \in \mathcal{W}^{q} $ are homogeneous elements.\\

For a fixed $a\in (\mathcal{W}^{q})_{i},$ we obtain the following  $q$-derivation
\begin{eqnarray*}
\varphi_{a}:&\mathcal{W}^{q}&\longrightarrow \mathcal{W}^{q}\\
&x&\longmapsto[a,x].
\end{eqnarray*}
The map  is denoted by $ad_{a}$ and is called the inner $q$-derivation.
\end{defn}

\begin{prop}
If $\varphi$ is an odd $q$-derivation of degree $s$ then it  is an inner
$q$-derivation, more precisely: $$\displaystyle (Der \mathcal{W}^{q} )_{1}=\oplus_{s\in \mathbb{Z}}\langle ad_{G_{s}}\rangle .$$
\end{prop}

\begin{proof}
\begin{eqnarray}
\textrm{Let }&\varphi&\textrm{ be an odd $q$-derivation of degree }  s, \varphi (L_{n})=a_{s,n}G_{n+s}  \textrm{ and } \varphi (G_{n})=b_{s,n}L_{n+s}.\label{odd}
\end{eqnarray}
 \textbf{Case 1: $s\neq -1$}\\
By (\ref{crochet1}) and (\ref{odd}), we have
\begin{eqnarray*}
  \{n\}\varphi (L_{n}) &=&\varphi ([L_{0},L_{n}])\\
   &=&\frac{1}{1+q^{s+1}} \Big([\varphi(L_{0}),\alpha(L_{n})]+[\alpha(L_{0}),\varphi(L_{n})]\Big) \\
   &=&  \frac{1}{1+q^{s+1}} \Big([a_{s,0}G_{s},(1+q^{n})L_{n}]+[2L_{0},a_{s,n}G_{s+n})]\Big)\\
   &=&\frac{1+q^{n}}{1+q^{s+1}} (\{n\}-\{s+1\})a_{s,0}G_{n+s}+2 a_{s,n}\frac{1}{1+q^{s+1}}\{n+s+1\}G_{n+s}.
\end{eqnarray*}
 We deduce that, when $s\neq -1,\ a_{s,n}=\frac{q^{s+1}-q^{n}}{q^{s+1}-1}a_{s,0}.$ On the other hand,
\begin{eqnarray*}
-\frac{a_{s,0}}{\{s+1\}}ad_{G_{s}}(L_{n})&=&-\frac{a_{s,0}}{\{s+1\}}[G_{s},L_{n}]\\
&=&\frac{a_{s,0}}{\{s+1\}}(\{s+1\}-\{n\})G_{n+s}\\
&=&\frac{q^{n}-q^{s+1}}{1-q^{s+1}}a_{s,0}G_{n+s}\\
&=&a_{s,n}G_{n+s}.
\end{eqnarray*}
So $\varphi (L_{n})=-\frac{a_{s,0}}{\{s+1\}}ad_{G_{s}}(L_{n}).$\\

By (\ref{crochet2}) and  (\ref{odd}), we have
\begin{eqnarray*}
  \{n+1\}\varphi (G_{n}) &=&\varphi ([L_{0},G_{n}])\\
   &=&\frac{1}{1+q^{s+1}} \Big([\varphi(L_{0}),\alpha(G_{n})]+[\alpha(L_{0}),\varphi(G_{n})]\Big) \\
   &=&  \frac{1}{1+q^{s+1}} \Big([a_{s,0}G_{s},(1+q^{n+1})G_{n}]+[2L_{0},b_{s,n}L_{s+n})]\Big)\\
   &=& 2 b_{s,n}\frac{1}{1+q^{s+1}}\{n+s\}L_{n+s}.
\end{eqnarray*}
We deduce that  $ \{n+1\} b_{s,n}=2 b_{s,n}\frac{1}{1+q^{s+1}}\{n+s\}$ so $b_{s,n}=0.$ Moreover,
\begin{eqnarray*}
-\frac{a_{s,0}}{\{s+1\}}ad_{G_{s}}(G_{n})&=&-\frac{a_{s,0}}{\{s+1\}}[G_{s},G_{n}]\\
&=&0\\
&=&b_{s,n}G_{n+s}\\&=&\varphi (G_{n})
\end{eqnarray*}
which implies   in this case 
$\varphi =-\frac{a_{s,0}}{\{s+1\}}ad_{G_{s}}$.\\



\textbf{Case 2 $s=-1$}\\
By (\ref{crochet1}) and (\ref{odd}), we have
\begin{eqnarray*}
  (\{m\}-\{n\})\varphi (L_{m+n}) &=& \varphi ([L_{n},L_{m}]) \\
  &=&\frac{1}{2} \Big([\varphi(L_{n}),\alpha(L_{m})]+[\alpha(L_{n}),\varphi(L_{m})]\Big) \\
   &=&  \frac{1}{2} \Big([a_{-1,n}G_{n-1},(1+q^{m})L_{m}]+[(1+q^{n})L_{n},a_{-1,m}G_{m-1}]\Big)\\
   &=& - \frac{1+q^{m}}{2}a_{-1,n} (\{n\}-\{m\})G_{m+n-1}+ \frac{1+q^{n}}{2}a_{-1,m} (\{m\}-\{n\})G_{m+n-1}.
\end{eqnarray*}
Then
\begin{eqnarray*}
 (\{m\}-\{n\})a_{-1,n+m}&=&- \frac{1+q^{m}}{2}a_{-1,n} (\{n\}-\{m\})+ \frac{1+q^{n}}{2}a_{-1,m} (\{m\}-\{n\}).
\end{eqnarray*}
So for $m \neq n$ we have

\begin{eqnarray}
a_{-1,n+m}&=&\frac{1+q^{m}}{2}a_{-1,n}+\frac{1+q^{n}}{2}a_{-1,m}. \label{derwitt1}
\end{eqnarray}
Letting $m = 0$\ in (\ref{derwitt1}), we obtain $\ a_{-1,0}=0. $\\
Letting $\  m = 1\ ,\ n=4$  in (\ref{derwitt1}), then
\begin{eqnarray}
a_{-1,5}&=&\frac{1+q}{2}a_{-1,4}+\frac{1+q^{4}}{2}a_{-1,1}. \label{derwitt2}
\end{eqnarray}
Letting $\  m = 1\ ,\ n=3$  in (\ref{derwitt1}), we obtain
\begin{eqnarray}
a_{-1,4}&=&\frac{1+q}{2}a_{-1,3}+\frac{1+q^{3}}{2}a_{-1,1}. \label{derwitt3}
\end{eqnarray}

Letting $\  m = 1\ ,\ n=2$  in (\ref{derwitt1}), we obtain
\begin{eqnarray}
a_{-1,3}&=&\frac{1+q}{2}a_{-1,2}+\frac{1+q^{2}}{2}a_{-1,1}. \label{derwitt4}
\end{eqnarray}
We deduce that

\begin{eqnarray}
 a_{-1,5}&=&\Big(\frac{1+q^{4}}{2}+\frac{1+q}{2}\frac{1+q^{3}}{2}+(\frac{1+q}{2})^{2}\frac{1+q^{2}}{2}\Big)a_{-1,1}
 +(\frac{1+q}{2})^{3}a_{-1,2}. \label{derwitt5}
 \end{eqnarray}
 Now, Letting $\  m = 2\ ,\ n=3$  in (\ref{derwitt1}), we obtain
 \begin{eqnarray}
  a_{-1,5}&=&\frac{1+q^{2}}{2}a_{-1,3}+\frac{1+q^{3}}{2}a_{-1,2}.\label{derwitt6}
  \end{eqnarray}

 By (\ref{derwitt4}) and (\ref{derwitt6}), we deduce that

 \begin{eqnarray}
 a_{-1,5}&=&(\frac{1+q^{2}}{2})^{2}a_{-1,1}+\Big(\frac{1+q}{2}\frac{1+q^{2}}{2}+\frac{1+q^{3}}{2}\Big)
a_{-1,2}.\label{derwitt7}
\end{eqnarray}

 Then, we deduce (by (\ref{derwitt5}) and (\ref{derwitt7})) that $a_{-1,2}=(1+q)a_{-1,1}=\{2\}a_{-1,1}.$\\
 Letting $\  m = 1\ $  in (\ref{derwitt1}), we obtain $a_{-1,n+1}=\frac{1+q^{1}}{2}a_{-1,n}+\frac{1+q^{n}}{2}a_{-1,1},$ and  by induction, we can show that
$a_{-1,n}=\{n\}a_{-1,1}.$ \\

So, $\varphi (L_{n})=\{n\}a_{-1,1}G_{n-1}
=a_{-1,1}[G_{-1},L_{n}],$ therefore
 \begin{eqnarray}
\varphi (L_{n})&=a_{-1,1}ad_{G_{-1}}(L_{n}).\label{derwithN}
\end{eqnarray}
Now, we calculate $\varphi (G_{n}):$
by (\ref{crochet2}) and (\ref{odd})
we have
\begin{eqnarray*}
(\{m+1\}-\{n\}) \varphi (G_{n+m})
   &=&\frac{1}{2} \Big([\varphi(L_{n}),\alpha(G_{m})]+[\alpha(L_{n}),\varphi(G_{m})]\Big) \\
   &=&  \frac{1}{2} \Big([a_{-1,n}G_{n-1},(1+q^{m+1})G_{m}]+[(1+q^{n})L_{n},b_{-1,m}L_{m-1})]\Big)\\
   &=&  b_{-1,m}\frac{1+q^{n}}{2} (\{m-1\}-\{n\})L_{m+n-1}.
\end{eqnarray*}
We  deduce that
\begin{eqnarray*}
(\{m+1\}-\{n\})b_{-1,m+n}&=& b_{-1,m}\frac{1+q^{n}}{2} (\{m-1\}-\{n\}).
\end{eqnarray*}

So for $m +1\neq n$ we have

\begin{eqnarray}
b_{-1,n+m}&=&\frac{1+q^{n}}{2}\frac{q^{n}-q^{m-1}}{q^{n}-q^{m+1}}b_{-1,m}.\label{derwitt8}
\end{eqnarray}
Letting  $m = 0$\ in (\ref{derwitt8}) (so   $n\neq 1$), \ we obtain

\begin{eqnarray}
b_{-1,n}&=&\frac{1+q^{n}}{2}\frac{q^{n}-q^{-1}}{q^{n}-q}b_{-1,0}.\label{derwitta}
\end{eqnarray}
So
\begin{eqnarray}
b_{-1,n+m}&=&\frac{1+q^{n+m}}{2}\frac{q^{n+m}-q^{-1}}{q^{n+m}-q}b_{-1,0}\label{derwitt9}
\end{eqnarray}
and
\begin{eqnarray}
b_{-1,m}&=&\frac{1+q^{m}}{2}\frac{q^{m}-q^{-1}}{q^{m}-q}b_{-1,0}.\label{derwittb}
\end{eqnarray}

By (\ref{derwitt8}) and (\ref{derwitt9}),   we have
\begin{eqnarray*}
\frac{1+q^{n+m}}{2}\frac{q^{n+m}-q^{-1}}{q^{n+m}-q}b_{-1,0}&=&\frac{1+q^{n}}{2}
\frac{q^{n}-q^{m-1}}{q^{n}-q^{m+1}}b_{-1,m}.
\end{eqnarray*}
If 
we replace $b_{-1,m}$ with its value given in (\ref{derwittb}), we obtain
\begin{eqnarray}
\frac{1+q^{n+m}}{2}\frac{q^{n+m}-q^{-1}}{q^{n+m}-q}b_{-1,0}
&=& \frac{1+q^{n}}{2}\frac{q^{n}-q^{m-1}}{q^{n}-q^{m+1}}\frac{1+q^{m}}{2}\frac{q^{m}
-q^{-1}}{q^{m}-q} b_{-1,0}.\label{derwitt10}
\end{eqnarray}
Letting $n=2,\ m=-3$ in (\ref{derwitt10}), we obtain $b_{-1,0}=0.$ By (\ref{derwitta}) we deduce that $b_{-1,n}=0,\ forall\  n\neq 1.$\\
Letting $m=1$ in (\ref{derwitt8}), we obtain
\begin{eqnarray*}
b_{-1,n+1}=\frac{1+q^n}{2}\frac{q^n-1}{q^n-q^{2}}b_{-1,1}.
\end{eqnarray*}

 We deduce that
\begin{eqnarray}
b_{-1,4}=\frac{1+q^3}{2}\frac{q^3-1}{q^3-q^{2}}b_{-1,1}.
\end{eqnarray}
So
 $b_{-1,1}=0.$

Since $b_{-1,n}=0\ forall\  n\neq 1$ and $b_{-1,1}=0,$ then $\varphi (G_{n})=0,\ forall\  n\in \mathbb{Z}.$

Since
\begin{eqnarray*}
\varphi (G_{n})&=&0
=a_{-1,1}[G_{-1},G_{n}],
\end{eqnarray*}
then
\begin{eqnarray}
\varphi (G_{n})&=&a_{-1,1}ad_{G_{-1}}(G_{n}).\label{derwith}
\end{eqnarray}
By (\ref{derwith}) and (\ref{derwithN}),
we deduce that $$\varphi =a_{-1,1}\ ad_{G_{-1}}.$$
\end{proof}

\begin{prop}
If $\varphi$ is an even $q$-derivation
 of degree $s$ then it is an inner
derivation, more precisely:$$\displaystyle (Der \mathcal{W}^{q} )_{0}=\oplus_{s\in \mathbb{Z}}\langle ad_{L_{s}}\rangle. $$
\end{prop}
\begin{proof}
\begin{eqnarray}
\textrm {Let } \varphi \textrm { be an even $q$-derivation
 of degree} \ s, \ \varphi (L_{n})=a_{s,n}L_{n+s}, \textrm { and } \varphi (G_{n})=b_{s,n}G_{n+s}.\label{eveder1}
\end{eqnarray}

 \textbf{Case 1: $s\neq 0$}\\
 By (\ref{crochet1}) and (\ref{eveder1}), we have
\begin{eqnarray*}
  \{n\}\varphi (L_{n}) &=&\varphi ([L_{0},L_{n}])\\
   &=&\frac{1}{1+q^{s}} \Big([\varphi(L_{0}),\alpha(L_{n})]+[\alpha(L_{0}),\varphi(L_{n})]\Big) \\
   &=&  \frac{1}{1+q^{s}} \Big([a_{s,0}L_{s},(1+q^{n})L_{n}]+[2L_{0},a_{s,n}L_{s+n})]\Big)\\
   &=&\frac{1+q^{n}}{1+q^{s}} (\{n\}-\{s\})a_{s,0}L_{n+s}+2 a_{s,n}\frac{1}{1+q^{s}}\{n+s\}L_{n+s}.
\end{eqnarray*}
 We deduce that, when $s\neq 0,\ a_{s,n}=\frac{q^{s}-q^{n}}{q^{s}-1}a_{s,0}.$ Moreover,

\begin{eqnarray*}
-\frac{a_{s,0}}{\{s\}}ad_{L_{s}}(L_{n})&=&-\frac{a_{s,0}}{\{s\}}[L_{s},L_{n}]\\
&=&-\frac{a_{s,0}}{\{s\}}(\{n\}-\{s\})L_{n+s}\\
&=&-\frac{q^{s}-q^{n}}{1-q^{s}}a_{s,0}L_{n+s}\\
&=&a_{s,n}L_{n+s}.\\
\end{eqnarray*}
So
\begin{eqnarray}
 \varphi (L_{n})&=&-\frac{a_{s,0}}{\{s\}}ad_{L_{s}}(L_{n}).\label{evender2}\\ \nonumber
\end{eqnarray}

Applying the same relations (\ref{crochet1}) and (\ref{eveder1}), we obtain
\begin{eqnarray*}
  \{n+1\}\varphi (G_{n}) &=&\varphi ([L_{0},G_{n}])\\
   &=&\frac{1}{1+q^{s}} \Big([\varphi(L_{0}),\alpha(G_{n})]+[\alpha(L_{0}),\varphi(G_{n})]\Big) \\
   &=&  \frac{1}{1+q^{s}} \Big([a_{s,0}L_{s},(1+q^{n+1})G_{n}]+[2L_{0},b_{s,n}G_{s+n})]\Big)\\
   &=&\frac{1+q^{n+1}}{1+q^{s}}a_{s,0}(\{n+1\}-\{s\}) L_{n+s}+2 b_{s,n}\frac{1}{1+q^{s}}\{n+s+1\}L_{n+s}.
\end{eqnarray*}
We deduce that  $  b_{s,n}=a_{s,0}\frac{q^{s}-q^{n+1}}{q^{s}-1}$. On the other hand,
\begin{eqnarray*}
-\frac{a_{s,0}}{\{s\}}ad_{L_{s}}(G_{n})&=&-\frac{a_{s,0}}{\{s\}}[L_{s},G_{n}]\\
&=&-\frac{a_{s,0}}{\{s\}}(\{n+1\}-\{s\})G_{n+s}\\
&=&a_{s,0}\frac{q^{s}-q^{n+1}}{q^{s}-1}G_{n+s}\\
&=&b_{s,n}G_{n+s}.
\end{eqnarray*}
So
\begin{eqnarray}
\varphi (G_{n})&=&-\frac{a_{s,0}}{\{s\}}ad_{L_{s}}(G_{n}).\label{evender3}
\end{eqnarray}
Using (\ref{evender2}) and (\ref{evender3}), we deduce that
\textbf{$\varphi =-\frac{a_{s,0}}{\{s\}}ad_{L_{s}}$}.\\

\textbf{Case 2 $s=0$}\\ By  (\ref{crochet1}) and (\ref{eveder1}), we have
\begin{eqnarray*}
  (\{m\}-\{n\})\varphi (L_{m+n}) &=& \varphi ([L_{n},L_{m}]) \\
  &=&\frac{1}{2} \Big([\varphi(L_{n}),\alpha(L_{m})]+[\alpha(L_{n}),\varphi(L_{m})]\Big) \\
   &=&  \frac{1}{2} \Big([a_{0,n}L_{n},(1+q^{m})L_{m}]+[(1+q^{n})L_{n},a_{0,m}L_{m})]\Big)\\
   &=& a_{0,n} \frac{1+q^{m}}{2} (\{m\}-\{n\})L_{m+n}+a_{0,m} \frac{1+q^{n}}{2} (\{m\}-\{n\})L_{m+n}.
\end{eqnarray*}
This implies that $a_{0,m+n}(\{m\}-\{n\})=a_{0,n} \frac{1+q^{m}}{2} (\{m\}-\{n\})+a_{0,m} \frac{1+q^{n}}{2} (\{m\}-\{n\}).$\\
So  for $m \neq n$, we have
\begin{eqnarray}
a_{0,n+m}&=&\frac{1+q^{m}}{2}a_{0,n}+\frac{1+q^{n}}{2}a_{0,m}.\label{even1}
\end{eqnarray}

Letting $m = 0$\ in (\ref{even1}), we obtain $\ a_{0,0}=0. $\\
Letting $m = 1$\ in (\ref{even1}), we obtain $a_{0,n+1}=\frac{1+q}{2}a_{0,n}+\frac{1+q^{n}}{2}a_{0,1}.$ By induction, we prove that
$ a_{0,n}=\{n\}a_{0,1}.$ So $\varphi(L_{n})=\{n\}a_{0,1}L_{n},$ that is
\begin{eqnarray*}
\varphi(L_{n})&=&\{n\}a_{0,1}L_{n}
=a_{0,1}[L_{0},L_{n}].
\end{eqnarray*}
Which leads to
\begin{eqnarray}
\varphi(L_{n})&=&a_{0,1}ad_{L_{0}}(L_{n}). \label{evender4}
\end{eqnarray}
By (\ref{crochet2}) and (\ref{eveder1}), we have
\begin{eqnarray*}
  (\{m+1\}-\{n\})\varphi (G_{m+n}) &=&\varphi ([L_{n},G_{m}])\\
   &=&\frac{1}{2} \Big([\varphi(L_{n}),\alpha(G_{m})]+[\alpha(L_{n}),\varphi(G_{m})]\Big) \\
   &=&  \frac{1}{2} \Big([a_{0,n}L_{n},(1+q^{m+1})G_{m}]+[(1+q^{n})L_{n},b_{0,m}G_{m})]\Big)\\
   &=&a_{0,n}(\{m+1\}-\{n\}) \frac{1+q^{m+1}}{2}G_{m+n} +b_{0,m}\frac{1+q^{n}}{2} (\{m+1\}-\{n\})G_{m+n}.
\end{eqnarray*}
We deduce that $b_{0,m+n}=a_{0,n}(\{m+1\}-\{n\}) \frac{1+q^{m+1}}{2} +b_{0,m}\frac{1+q^{n}}{2} (\{m+1\}-\{n\}).$\\
So, for $m +1\neq n$, it holds that
\begin{eqnarray}
b_{0,n+m}&=&a_{0,n}\frac{1+q^{m+1}}{2}+b_{0,m}\frac{1+q^{n}}{2}.\label{even2}
\end{eqnarray}
Taking $m = 0,$ in (\ref{even2}), we obtain $b_{0,n}=a_{0,n}\frac{1+q}{2}+b_{0,0}\frac{1+q^{n}}{2}.$  Since $a_{0,n}=a_{0,1}\{n\},$ then
\begin{eqnarray}
\ b_{0,n}&=&a_{0,1}\{n\}\frac{1+q}{2}+b_{0,0}\frac{1+q^{n}}{2}.\label{even3}
\end{eqnarray}
Taking $m = 1,$ $n=-1 $ in (\ref{even2}) and $n=1$ in $(\ref{even3})$, we obtain $b_{0,0}=a_{0,1}.$ So  using (\ref{even3}), we get
\begin{eqnarray*}
\ b_{0,n}&=a_{0,1}&\{n\}\frac{1+q}{2}+b_{0,0}\frac{1+q^{n}}{2}\\
&=&a_{0,1}\frac{1-q^{n}}{1-q}\frac{1+q}{2}+a_{0,1}\frac{1+q^{n}}{2}\\
&=&a_{0,1}\{n+1\}.
\end{eqnarray*}
 Then
\begin{eqnarray*}
\varphi (G_{n})&=&b_{0,n}G_{n}
=a_{0,1}\{n+1\}G_{n}
=a_{0,1}[L_{0},G_{n}].
\end{eqnarray*} Therefore
\begin{eqnarray}
\varphi (G_{n})&=&a_{0,1} \ ad_{L_{0}}(G_{n}).\label{evender5}
\end{eqnarray}
 By (\ref{evender4}) and (\ref{evender5}), we deduce that
\textbf{$\varphi =a_{0,1} \ ad_{L_{0}}$}.\\
\end{proof}

\subsection{Cohomology space $H_{r,0}^{2}(\mathcal{W}^{q})$ of  $\mathcal{W}^{q}$}
In the  following we describe the  cohomology space $H_{0}^{2}(\mathcal{W}^{q},\mathbb{C})$. We denote by $[f]$ the cohomology class of an element $f$.

\begin{thm}
$$H_{r,0}^{2}(\mathcal{W}^{q})=\mathbb{C}[\varphi_{1}]\oplus\mathbb{C}[\varphi_{2}],
$$ where
\begin{eqnarray*}
\varphi_{1}(xL_{n}+yG_{m},zL_{p}+tG_{k})&=&xzb_{n}\delta_{n+p,0},\\
\varphi_{2}(xL_{n}+yG_{m},zL_{p}+tG_{k})&=&xtb_{n}\delta_{n+k,-1}-yz b_{p}\delta_{p+m,-1},\\
\end{eqnarray*}
with
\begin{eqnarray*}
b_{n}&=&\frac{1}{q^{n-2}}\frac{1+q^{2}}{1+q^{n}}\frac{(1-q^{n+1})(1-q^{n})(1-q^{n-1})}{(1-q^{3})(1-q^{2})(1-q)}.
\end{eqnarray*}

\end{thm}
\begin{proof}

For all $f\in C_{\alpha,Id_{\mathbb{C}}}^{2}(\mathcal{W}^{q},\mathbb{C}),$  we have (see (\ref{def cob}))
\begin{eqnarray}
\delta(f)(x_{0},x_{1},x_{2})=-f([x_{0},x_{1}],\alpha(x_{2}))+(-1)^{|x_{2}||x_{1}|}f([x_{0},x_{2}],\alpha(x_{1}))+f(\alpha(x_{0}),[x_{1},x_{2}]) .\ \ \ \ \, \label{ddd}
\end{eqnarray}
Now, suppose that $f$ is a $q$-deformed 2-cocycle on $\mathcal{W}^q.$  From (\ref{ddd}), we obtain
\begin{eqnarray}
-f([x_{0},x_{1}],\alpha(x_{2}))+(-1)^{|x_{2}||x_{1}|}f([x_{0},x_{2}],\alpha(x_{1}))+f(\alpha(x_{0}),[x_{1},x_{2}])&=&0 .\label{ddd1}
\end{eqnarray}
By (\ref{ddd1})  and
taking the triple $(x, y, z)$ to be $(L_{n}, L_{m}, L_{p}),\ (L_{n}, L_{m},G_{p}),\  and \ (L_{n},G_{m} ,G_{p})$, respectively, we obtain $f(L_{n}, L_{p}),\ f(L_{n},G_{p})$ and $f(G_{n}, G_{p})$ which define $f.$

\textbf{Case 1}: $X=(L_{n},L_{m},L_{p})$\\
Using (\ref{ddd1}), we have
\begin{eqnarray*}
-f([L_{n}, L_{m}],\alpha(L_{p}))+f([L_{n}, L_{p}],\alpha(L_{m}))+f(\alpha(L_{n}),[L_{m}, L_{p}])&=&0 .
\end{eqnarray*}
Since $[L_{m},L_{p}]=(\{p\}-\{m\}) L_{m+p}$ , and $\alpha(L_{n})=(1+q^{n})L_{n}$   then
\begin{eqnarray}
-(1+q^{p})(\{m\}-\{n\})f(L_{n+m},L_{p})+ (1+q^{m})(\{p\}-\{n\})f(L_{n+p},L_{m})\nonumber\\+(1+q^{n})(\{p\}-\{m\})f(L_{n},L_{m+p})=0.\label{homm1}
\end{eqnarray}

Setting $m=0$, in (\ref{homm1}), we obtain $\displaystyle f(L_{n},L_{p})=\frac{q^{n}-q^{p}}{1-q^{n+p}} f(L_{0},L_{n+p})\ \  (n+p\neq 0).$\\
Setting $m=0,\ n=-p$, in (\ref{homm1}), then we obtain $ f(L_{0},L_{0})=0.$\\
Setting $m=-n-p$  in (\ref{homm1}), we obtain
\begin{eqnarray}
-(1+q^{p})(q^{n}-q^{-n-p})f(L_{-p},L_{p})+(1+q^{-n-p})(q^{n}-q^{p})f(L_{n+p},L_{-n-p})\nonumber \\ +(1+q^{n})(q^{-n-p}-q^{p})f(L_{n},L_{-n})=0.
\label{homm2}
\end{eqnarray}
Setting $p=1$  (\ref{homm2}), we obtain
\begin{eqnarray}
-(1+q)(q^{2n+1}-1)f(L_{-1},L_{1})+q(1+q^{n+1})(q^{n-1}-1)f(L_{n+1},L_{-n-1})\nonumber\\+(1+q^{n})(1-q^{n+2})f(L_{n},L_{-n})=0.
\label{homm2s}
\end{eqnarray}
Hence,
\begin{equation}\label{19avril}
 f(L_{n+1},L_{-n-1})=\frac{1}{q}\frac{1+q^{n}}{1+q^{n+1}} \frac{1-q^{n+2}}{1-q^{n-1}}f(L_{n},L_{-n})-\frac{1}{q}\frac{1+q}{1+q^{n+1}}\frac{1-q^{2n+1}}{1-q^{n-1}}f(L_{1},L_{-1}),\  for \ n\neq 1,
\end{equation}
\begin{equation}\label{19april}
f(L_{n},L_{-n})=q\frac{1+q^{n+1}}{1+q^{n}}\frac{1-q^{n-1}}{1-q^{n+2}} f(L_{n+1},L_{-n-1})+\frac{1+q}{1+q^{n}}\frac{1-q^{2n+1}}{1-q^{n+2}}f(L_{1},L_{-1}),\  for \ n\neq -2.
\end{equation}
Letting
\begin{equation*}
 \alpha_{n}= \frac{1}{q}\frac{1+q^{n-1}}{1+q^{n}} \frac{1-q^{n+1}}{1-q^{n-2}} \ and \
    \beta_{n}=-\frac{1}{q}\frac{1+q}{1+q^{n}}\frac{1-q^{2n-1}}{1-q^{n-2}}
\end{equation*}

From the  formula (\ref{19avril}) we get



 $f(L_{n},L_{-n})=a_{n}f(L_{1},L_{-1})+b_{n}f(L_{2},L_{-2}),$ for $n>2$,
 where
\begin{eqnarray*}
    a_{n}=\beta_{n}+\alpha_{n}\beta_{n-1}+\alpha_{n}\alpha_{n-1}\ \beta_{n-2}+\cdots+\alpha_{n}\alpha_{n-1}\cdots \alpha_{4}    \beta_{3},
\end{eqnarray*}
  and $$b_{n}= \frac{1}{q^{n-2}}\frac{1+q^{2}}{1+q^{n}}\frac{(1-q^{n+1})(1-q^{n})(1-q^{n-1})}{(1-q^{3})(1-q^{2})(1-q)}.$$
Letting
\begin{equation*}
 \alpha_{n}'= q\frac{1+q^{n+1}}{1+q^{n}}\frac{1-q^{n-1}}{1-q^{n+2}}  \ and \
    \beta_{n}'=q\frac{1+q}{1+q^{n}}\frac{1-q^{2n+1}}{1-q^{n+2}}.
\end{equation*}

From the  formula (\ref{19april}) we get
 $f(L_{n},L_{-n})=a_{n}'f(L_{-1},L_{1})+b_{n}'f(L_{-2},L_{2}),$ for $n<-2$,
 where
\begin{eqnarray*}
    a_{n}'=\beta_{n}'+\alpha_{n}\beta_{n+1}'+\alpha_{n}'\alpha_{n+1}'\ \beta_{n+2}'+\cdots+\alpha_{n}'\alpha_{n+1}'\cdots \alpha_{-4}'    \beta_{-3}',
\end{eqnarray*}

\begin{equation*}
 b_{n}' =\frac{1}{q^{n+2}}\frac{1+q^{-2}}{1+q^{n}}\frac{(1-q^{n-1})(1-q^{n})(1-q^{n+1})}{(1-q^{-3})(1-q^{-2})(1-q^{-1})}=-b_{-n}.
\end{equation*}

\textbf{Case 2}: $X=(L_{n}, L_{m},G_{p}).$\\
By (\ref{ddd1}), we have
\begin{eqnarray*}
-f([L_{n},L_{m}],\alpha( G_{p}))+f([L_{n},G_{p}],\alpha(L_{m}))+f(\alpha(L_{n}),[L_{m},G_{p}])&=&0.
\end{eqnarray*}
Since $[L_{n},G_{p}]=(\{p+1\}- \{n\})G_{n+p}$ and $\alpha(G_{n})=(1+q^{n+1})G_{n}$, then

\begin{eqnarray}
-(1+q^{p+1})(\{m\}- \{n\})f(L_{n+m},G_{p})+(1+q^{m})(\{p+1\}- \{n\})f(G_{n+p},L_{m})+\nonumber\\ (1+q^{n})(\{p+1\}- \{m\}) f(L_{n},G_{m+p})
   =0.\label{homm3}
\end{eqnarray}
Taking $m=0$ in (\ref{homm3}), we obtain
\begin{equation}\label{20april}
 (1-q^{n+p+1})f(L_{n},G_{p})=(q^{n}-q^{p+1})f(L_{0},G_{n+p}).
\end{equation}
Then
\begin{equation*}
  f(L_{n},G_{p})=\frac{q^{n}-q^{p+1}}{1-q^{n+p+1}}f(L_{0},G_{n+p})\ \textrm{for} \ n+p+1 \neq 0.
\end{equation*}
Taking $n=1,\ p=-2$ in (\ref{20april}), we obtain $f(L_{0},G_{-1})=0.$\\
Taking $\ m=-n,\ p=-1$ in (\ref{homm3}),  we obtain (with $f(L_{0},G_{-1})=0$)
\begin{equation*}
f(L_n,G_{-n-1})=-f(L_{-n},G_{n-1}).
\end{equation*}
Then $f(L_1,G_{-2})=-f(L_{-1},G_{0})$,\ $f(L_2,G_{-3})=-f(L_{-2},G_{1}). $\\
Taking $m=-1,\ p=-n\ $ in (\ref{homm3}), we obtain
\begin{eqnarray*}
-(1+q^{n-1})(q^{n+1}-1)f(L_{n-1},G_{-n})+(1+q)(q^{2n-1}- 1)f(G_{0},L_{-1})+\\q(1+q^{n})(q^{n-2}- 1) f(L_{n},G_{-n-1})
   =0.
\end{eqnarray*}
Hence
\begin{eqnarray}
 f(L_{n},G_{-n-1})=\frac{1}{q}\frac{1+q^{n-1}}{1+q^{n}}\frac{1-q^{n+1}}{1-q^{n-2}}f(L_{n-1},G_{-n}) -\frac{1}{q}\frac{1+q}{1+q^{n}}
 \frac{1-q^{2n-1}}{1-q^{n-2}}f(L_{1},G_{-2})\nonumber \\  \textrm{for}\  n\neq 2\ , \label{nuitnuit}
\end{eqnarray}
\begin{eqnarray}
f(L_{n-1},G_{-n})=q\frac{1+q^{n}}{1+q^{n-1}} \frac{1-q^{n-2}}{1-q^{n+1}} f(L_{n},G_{-n-1})+\frac{1+q}{1+q^{n-1}}  \frac{1-q^{2n-1}}{1-q^{n+1}}f(L_{-1},G_{0}) \nonumber \\  \textrm{for}\  n\neq -1.\label{20avril}
\end{eqnarray}
  Comparing (\ref{19avril}) and (\ref{nuitnuit}), we deduce that
 \begin{equation*}
f(L_{n},G_{-n-1})=a_{n}f(L_{1},G_{-2})+b_{n}f(L_{2},G_{-3})\  \textrm{for}\  n>2.
 \end{equation*}
 Comparing (\ref{19april}) and (\ref{20avril}),  we deduce that
  \begin{equation*}
f(L_{n},G_{-n-1})=a_{n}'f(L_{-1},G_{0})+b_{n}'f(L_{-2},G_{1})\  \textrm{for}\  n<-2,
 \end{equation*}
where $a_{n},\ b_{n},\ a_{n}'$ and $b_{n}'$ are defined as in the previous case.\\
\textbf{Case 3}: $X=(L_{n},\ G_{m},\ G_{p}).$\\
By (\ref{ddd1}), we have

\begin{eqnarray*}
-f([L_{n},G_{m}],\ \alpha(G_{p}))-f([L_{n},G_{p}],\ \alpha(G_{m}))+f(\alpha(L_{n}),[G_{m},G_{p}])=0.
\end{eqnarray*}
So
\begin{eqnarray}
-(1+q^{p+1}) (\{m+1\}- \{n\})f(G_{m+n},G_{p})-(1+q^{m+1}) (\{p+1\}\\- \{n\})f(G_{p+n},G_{m})=0.\label{homm4}
\end{eqnarray}
Taking $m=0,\ $ in (\ref{homm4}), we obtain
\begin{eqnarray}
 (1+q^{p+1}) (\{1\}- \{n\})f(G_{n},G_{p})+(1+q) (\{p+1\}- \{n\})f(G_{p+n},G_{0})&=&0.\label{homm5}
\end{eqnarray}
Taking $  n=1 $ and replacing $p+1$ by $k$ in (\ref{homm5}), we obtain
\begin{equation*}
f(G_{k},G_{0})=0\ \textrm{for}\ k \neq1.  
\end{equation*}
Hence,
\begin{equation*}
  f(G_{n},G_{p}) =0   \ \textrm{for}\ n \neq1,\ p+n\neq 1.
\end{equation*}
Taking  $p=1-n$ in   (\ref{homm5}),  we obtain
\begin{equation}\label{avril}
  f(G_{n},G_{1-n})=-\frac{1+q}{1+q^{2-n}}(1+q^{1-n})f(G_{1},G_{0}) \ \ (n\neq 1).
\end{equation}
Replacing $n$ by $1-n$ and $ p$ by $n$ in (\ref{homm5}), one can obtain
\begin{equation}
 f(G_{1-n},G_{n})=-\frac{1+q}{1+q^{n+1}}(1+q^{n})f(G_{1},G_{0}) \ \ (n\neq 0).
\end{equation}
Then using  super skew-symmetry of $f$, one can obtain $f(G_{1},G_{0})=0$.

We deduce that $f(G_{n},G_{m})=0,\ \forall n,\ m\in \mathbb{Z}$.


We denote by $g$ the linear map defined on   $\mathcal{W}^{q}$   by
\begin{eqnarray*}
 g(L_{n})  &=&-\frac{1}{\{n\}}f(L_{0},L_{n})\ \textrm{if $n\neq 0,$}\
 g(L_{0})=-\frac{q}{q+1}f(L_{1},L_{-1}),\\  g(G_{n}) & =&\frac{1}{\{n+1\}}f(L_{0},G_{n})  \textrm{ if $n\neq -1,\ $}\  g(G_{-1})=-\frac{q}{q+1}f(L_{1},G_{-2}).
\end{eqnarray*}
It is easy to verify that $\displaystyle \delta(g)(L_{n},L_{p})=\frac{q^{p}-q^{n}}{1-q^{n+p}}f(L_{0},L_{n+p}) \ (p\neq -n),\\ \delta(g)(L_{n},L_{-n})=
0, \ \delta(g)(L_{n},G_{p})=\frac{q^{p+1}-q^{n}}{1-q^{p+n+1}}f(L_{0},G_{n+p})\ (p+n\neq -1)$ \\and $\delta(g)(G_{n},G_{p})=0.$\\
Let $h=f-\delta^{1} g$. Then we have
\begin{eqnarray*}
  h(L_{1},L_{-1})&=&h(L_ {1} ,G_{ -2})=0,\\
  h(L_n,L_p) &=& 0 \ \textrm{for}\ n+p\neq0,\\
  h(L_n,G_p) &=& 0 \ \textrm{for}\   n+p\neq -1,\\
  h(G_{n},G_{p})&=&0 \ \textrm{for all}\  n,\ p\in\mathbb{Z}.
\end{eqnarray*}
Since $h$ is a 2-cocycle we deduce that:
\begin{eqnarray*}
  h(L_{n},L_{-n})&=&a_{n}h(L_{1},L_{-1})+b_{n}h(L_{2},L_{-2})=b_{n}h(L_{2},L_{-2}) \\
  h(L_n,G_{-n-1}&=&a_n h(L_ {1} ,G_{ -2})+b_nf(L_{2},G_{-3})=b_nh(L_{2},G_{-3}) \\
 h(G_n,G_m)  &=&0.
\end{eqnarray*}

Using the equalities above, we deduce that
\begin{eqnarray*}
&& h(xL_{n}+yG_{m}, zL_{p}+tG_{k})\\
 &&=   xz
 \delta_{n+p,0}b_{n}h(L_{2},L_{-2})
+xt\delta_{n+k,-1}b_{n}h(L_{2},G_{-3})
 -yz\delta_{p+m,-1} b_{p}    h(L_{2},G_{-3})\\
 &=&h(L_{2},L_{-2})\varphi_{1}(xL_{n}+yG_{m}, zL_{p}+tG_{k})+h(L_{2},G_{-3})\varphi_{2}(xL_{n}+yG_{m}, zL_{p}+tG_{k}),
 \end{eqnarray*}
that ends the proof.
\end{proof}

\begin{cor}
Let $(V,\beta)$ be  a $\mathcal{W}^q$-module and $f  \in C^2(\mathcal{W}^q,V)$. Define a bracket and a morphism on $\widetilde{\mathcal{W}^q }=\mathcal{W}^q\oplus V$ by
\begin{eqnarray*}
[(x,a),(y,b)]_{\widetilde{\mathcal{W}^q }}&=&([x,y],f(x,y))\\
\widetilde{\alpha}(x,a)&=&(\alpha(x),a)\ \ \forall x,\ y \in \mathcal{W}^q,\ a,\ b\in V.
\end{eqnarray*}
The triple $(\widetilde{\mathcal{W}^q },[.,.]_{\widetilde{\mathcal{W}^q }},\widetilde{\alpha })$ is a Hom-Lie superalgebra if and only if $f$ is in $\displaystyle \mathbb{C}[\varphi_{1}]\oplus \mathbb{C}[\varphi_{2}]$
\end{cor}

\end{document}